\documentclass[11pt]{article}   
\usepackage{amssymb,amscd,latexsym}   
\usepackage{amsmath}
\usepackage{amsthm}
\usepackage{mathdots}
\usepackage[pagebackref]{hyperref}
\usepackage{color}
\usepackage[all]{xy}
\usepackage{imakeidx}         
\usepackage{graphicx}        
\usepackage{graphicx,color} 
\usepackage{tikz}
       
\usepackage{graphicx,color} 
\usepackage{tikz}
\newcommand{\rar}{\rightarrow}
\newcommand{\lar}{\longrightarrow}
\newcommand{\llar}{-\kern-5pt-\kern-5pt\longrightarrow}
\newcommand{\surjects}{\twoheadrightarrow}

\newtheorem{Theorem}{Theorem}[section]
\newtheorem{Lemma}[Theorem]{Lemma}
\newtheorem{Corollary}[Theorem]{Corollary}
\newtheorem{Proposition}[Theorem]{Proposition}
\newtheorem{Remark}[Theorem]{Remark}
\newtheorem{Example}[Theorem]{Example}

\newtheorem{Conjecture}[Theorem]{Conjecture}
\newtheorem{Definition}[Theorem]{Definition}
\newtheorem{Question}[Theorem]{Question}
\newtheorem{Questions}[Theorem]{Questions}

\newtheorem{Setup}[Theorem]{Setup}
\def\sqr#1#2{{\vcenter{\hrule height.#2pt
        \hbox{\vrule width.#2pt height#1pt \kern#1pt
            \vrule width.#2pt}
        \hrule height.#2pt}}}
\def\phi{\varphi}

\DeclareMathOperator{\Image}{Im}
\DeclareMathOperator{\coker}{Coker}

\DeclareMathOperator{\Proj}{Proj}

\DeclareMathOperator{\depth}{depth}
\DeclareMathOperator{\rank}{rank}
\DeclareMathOperator{\Ht}{ht}

\DeclareMathOperator{\grade}{grade}

\def\xx{{\bf x}}
\def\yy{{\bf y}}

\def\TT{{\bf T}}
\def\tt{{\bf t}}

\def\XX{{\bf X}}
\def\YY{{\bf Y}}

\def\uu{{\bf u}}

\def\fm{{\mathfrak m}}
\def\fn{{\mathfrak n}}

\def\Ht{{\rm ht}\,}
\def\depth{{\rm depth}\,}

\def\codim{{\rm codim}\,}

\def\ker{{\rm ker}\,}
\def\grade{{\rm grade}\,}
\def\rk{{\rm rank}\,}
\def\syz{\mbox{\rm Syz}}
\def\spec#1{{\rm Spec}\, (#1)}

\def\restr{{\kern-1pt\restriction\kern-1pt}}

\def\pp{{\mathbb P}}

\begin{document}
\begin{center}
	{\large{\bf\sc  Matrices over polynomial rings approached by commutative algebra}}
	\footnotetext{AMS Mathematics
		Subject Classification (2010   Revision). Primary 13A02, 13A30, 13D02, 13H10, 13H15; Secondary  14E05, 14M07, 14M10,  14M12.} 
	
\medskip

{To the memory of J\"urgen Herzog}

\medskip
	
		{\sc Zaqueu Ramos}\footnote{Partially supported by a CNPq grant (304122/2022-0)}  \quad\quad
	{\large\sc Aron  Simis}\footnote{Partially
		supported by a CNPq grant (301131/2019-8).}

\end{center}



\begin{abstract}
The main goal of the paper is the discussion of a deeper interaction between matrix theory over polynomial rings over a field and typical methods of commutative algebra and related algebraic geometry. This is intended in the sense of bringing numerical algebraic invariants into the picture of determinantal ideals, with an emphasis on non-generic ones. In particular, there is a strong focus on square sparse matrices and features of the dual variety to a determinantal hypersurface. Though the overall goal is not exhausted here, one provides several environments where the present treatment has a degree of success.
	
\end{abstract}

\section*{Introduction}

It is somewhat surprising  how little has been published, in greater depth,  about matrices and their minors over arbitrary commutative rings. We quote from a fairly recent paper by well-known experts (\cite{KodLamSwan}): ``a general algebraic study of the properties of
the determinantal ideals of a symmetric or skew-symmetric matrix over an arbitrary
commutative ring has remained relatively lacking".
It is not altogether absurd to note that such a study is also relatively lacking even over the familiar  rings of commutative algebra and their  avatars in algebraic geometry.
The gross amount of matrix based theorems in  commutative algebra relate to the case where the entries are indeterminates over a field or closely derived. 

On the other hand, there has been a good amount of work when the entries of the matrix are homogeneous linear forms over a field -- such matrices being  named {\em linear sections} of the generic matrix. Important contribution along this line has been done in \cite{Eisenbud2}, based on the assumption of $k$-genericity, a property of linear sections that partially replace the absence of fully indeterminate entries.
At a special corner, some of the current work has focused on {\em coordinate} linear sections, whereby the entries are either indeterminates or zeros (see, e.g., \cite{HT}, \cite{Mer-Gius},  \cite{CRS}, \cite{Maral}, \cite{Degen-Gen}, \cite{Deg_sym}). By common acceptance, if zeros actually come up, one informally call such matrices  {\em sparse}.

Throughout, the objective is a discussion of what could be named {\em algebraic matrix theory} based upon a substantial use of some of the successful invariants of commutative algebra, sidelong, algebraic geometry.
Determinantal theory (ideals, varieties) is the first thing that comes up to our minds as related to the title of this work. Although typical of this relation, it is not exclusive. Some of the topics over here were inspired by parts of the book \cite{DetBook2024}, but conversely, quite a bit there was an expansion of the first ideas here.
Of the four sections of this paper, the last three are entirely new research, while the first is an assemblage of some algebraic results scattered in \cite{DetBook2024} and elsewhere before, making strong use of the adjugate of a square matrix.

A brief word on the sections.

As said, Section~\ref{Sec1} is an invitation to the use of the adjugate in typical problems arising from the algebraic and geometric aspects of matrices and their ideals of minors. It covers the role of $2\times 2$-minors upon the dual variety of a hypersurface to an encore of the dimension of the Grassmann algebra in arbitrary characteristic to the structure of submaximal minors of a certain ladder to the homaloidal properties of certain sparse symmetric matrices.

The second section deals with certain square matrices of non-maximal rank, which forces us to focus upon their ideals of lower minors. The gist of the theme is that such matrices often appear as syzygy matrices of non-perfect ideals of height $2$.
We isolate several environments in which these matrices naturally appear, where their relevance lies on the difficult nature of the structure of the ideals of which they are syzygies. Current examples of this environment are provided by the nearly-free hypersurfaces introduced in \cite[Definition 2.6]{Dimca_et_al2021}.

Section~\ref{Sec3} deals with square matrices with a particular sparse structure. The main results cover the structure of two main related rational maps, the first defined by the cofactors, and the second, by the partial derivatives of the determinant (polar map). As usual, for non-generic square matrices, there is a ``balanced'' case, which is fully developed, including the birationality of both the cofactor map and the polar map, along with the precise structure of the image of the former -- in particular, as it turns, an arithmetically  Gorenstein variety.

Section~\ref{HB} sort of stands alone regarding the previous sections. Here we deal permanently with $n\times (n-1)$ matrices -- often, as in here, called Hilbert--Burch matrices due to the well-known structural theorem associated with these names.
In particular, in contrast to the theme of Section~\ref{Sec2}, here the main ideal is still of codimension two but this time around perfect.
Under this assumption, ideals in two variables (called {\em binary} at times) are primary for the maximal ideals, so their theory is very special, but not trivial at all (see\cite{BuJou}, \cite{Cox07}, \cite{CHW}, \cite{Syl1}, a few among many others).
We assume, perhaps irresponsibly, that this case is well-understood, and basically assume dimension at least three.

The section is then divided in two subsections.

In the first subsection we focus on the case of dimension three and assume that the ideal of maximal minors is equigenerated.
Some of the results in this part are reminiscent of earlier discussions held by the senior author with W. Vasconcelos and J. Hong,   as a sequel to their joint work on  aspects concerning ternary elimination theory (\cite{Syl1}, \cite{Syl2}).
It is crucial that we are not assuming the entries of the matrices to be all linear, as otherwise, quite a bit has been said recently in this dimension (\cite{Lan}, \cite{linpres2018}).
We only bring out the discussion when $n=3$, accessing completely the case of one linear syzygy.
The case where both syzygies have degrees $\geq 2$ is contemplated via examples. The general expectation is that the homogeneous defining ideal of the associated Rees algebra is generated by Sylvester forms.
A couple of questions are posed at the end of the subsection in the belief that they are still open.

The second subsection takes the turn of an $n\times (n-1)$ matrix whose transpose is generated in two degrees $\epsilon_i, i=1,2$. Then, accordingly, the ideal $I$ of maximal minors is generated in two degrees.
The focus in on the subideal $J\subset I$ generated by the maximal minors fixing, say, the rows of degree $\epsilon_2$. 
This sort of ideal having been discussed before in \cite{AnSi1986} in ``abstract'',  is revisited here in this particular graded case by focusing upon the inclusion $J^{\rm sat}\subset I$. As such, the graded module $I/J$ plays a central role, so we mind to look at its graded free resolution by way of the Buchsbaum--Rim complex.
Two cases receive particular attention, the one where the number of rows of degree $\epsilon_1$ is  least (interesting) possible, namely, $3$, and the other where this degree is maximum, namely, the dimension of the ground polynomial ring.
In the latter case, assuming moreover that $\epsilon_1=\epsilon_2=1$ (linear case), we prove that $J$ is a reduction of $I$, a result that may be known in some disguised way but which we have not been able to find in this form.

Some of the foremost theorems in the paper are Theorem~\ref{adjugate_vs_2minors},
 Theorem~\ref{2-minors_vs_dual_variety}, 
 Theorem~\ref{rank_and_dimension},
 Theorem~\ref{inversion_factors_resolution},
 Theorem~\ref{ratmapminorssubmax},                                Theorem~\ref{Hessian_and_homaloid},
Theorem~\ref{is_a_domain}.

The standing reference here for basic commutative algebra is \cite{SimisBook}, while for steps into determinantal ideals we refer to \cite{DetBook2024}. 

\section{The ubiquity of the adjugate}\label{Sec1}

\subsection{Preamble}

By a matrix  { \em over a commutative ring $R$} is meant a matrix whose entries are elements of $R$, in which case we call $R$ the {\em ground ring} of the matrix. We have no particular notation for a matrix, as it will vary according to the appropriate section.
Often it will denoted $(a_{ij})_{1\leq i\leq m;1\leq j\leq n}$ as to emphasise its entries and size.

All rings in this work are commutative, so we will refrain from repeating it.
We hereby take for granted the basic properties of determinants.

Given a matrix $\mathcal{M}=(a_{ij})_{1\leq i\leq m;1\leq j\leq n}$ over a commutative ring $R$, picking two sets of $t\leq \min\{m,n\}$ indices each, $\{i_1,\ldots,i_t\}$ and $\{j_1,\ldots,j_t\}$, the determinant
$$\det
\left[\begin{matrix}
	a_{i_1j_1} & \ldots &a_{i_1j_t}\\
	\vdots & \ldots & \vdots\\
	a_{i_tj_1} & \ldots &a_{i_tj_t}\\
\end{matrix}
\right]\in R
$$
is called a {\em {\rm (}determinantal{\rm )} $t$-minor}, thus adhering to the established abuse of calling both the $t\times t$ submatrix and its determinant by the same name.
Following a quite standard notation in commutative algebra, we let $I_t(\mathcal{M})\subset R$ denote the ideal generated by the $t$-minors -- again, often informally called a {\em determinantal ideal}.

Quite clearly, the full nature of this ideal would not come up when taking the ground ring to be a field, except to decide whether it vanishes or not.
And yet, regardless of the nature of the ground ring, one of the beautiful properties of the ideal of minors is its invariance under conjugation, namely,
$I_t(\mathcal{M})=I_t(\mathcal{U}\mathcal{M}\mathcal{V}),$ where $\mathcal{U}, \mathcal{V}$ are invertible matrices over $R$ of order $m\times m$ and $n\times n$, respectively.

This implies that the $R$-linear map $\psi: R^n \rar R^m$  of free $R$-modules induced by $\mathcal{M}$ is well-defined independently of the chosen bases.
Hans Fitting then proved a second result to the effect that,  for any $t$, $I_t(\mathcal{M})$ depends only on the $R$-module $E:=\coker \psi$.
However, to make it work one has to focus instead on the $(e-t)$-minors of the matrix that comes up by choosing a set of $e$ generators of $E$.
As such,  these ideals of minors are  called {\em Fitting ideals} (of $E$) -- see \cite[Chapter 3, Section 3.3]{SimisBook} for a fuller account on this invariant.

A fundamental subsidiary concept to minors is that of a cofactor: given  indices $i,j$, the $(i,j)$-{\em cofactor} of $\mathcal{M}$ is the minor whose rows and columns are complementary to the $i$th rows and $j$th columns, respectively.
The cofactor is traditionally attributed a sign given by $(-1)^{i+j} $.
We will usually denote the $(i,j)$-cofactor  by $\Delta_{i,j}$, whenever the matrix is self-understood.

The notion of the adjugate to a square matrix goes back to the origins of the well-established theory of determinants, having been introduced by Cauchy in the early nineteenth century.
Yet, its usefulness goes well beyond the traditional uses in linear algebra. In commutative algebra it is a tool to approach various questions concerning matrices with homogeneous linear entries, such as the structure of the linear syzygies of ideals of cofactors, features of Hessian matrices and inverse maps to birational maps.
Some of these aspects will be dealt with in this paper.

Let $\mathcal{M}$ denote an $m\times m$ matrix over a ring $R$.

The {\em adjugate matrix} of $\mathcal{M}$ is the transpose of the $m\times m$ matrix whose entries are the (signed) cofactors of $\mathcal{M}$.
We denote it by ${\rm adj}(\mathcal{M})$.
Its basic structure property is the relation established by Cauchy:
\begin{equation}\label{adjugate_relation}
	\mathcal{M}\,{\rm adj}(\mathcal{M})={\rm adj}(\mathcal{M})\, \mathcal{M}=(\det \mathcal{M})\, \mathbb{I}_m,
\end{equation}
where $\mathbb{I}_m$ denotes the $m\times m$ identity matrix.

The adjugate is often called the {\em classical adjoint} or the {\em adjunct}, not to be confused with the adjoint operator, the latter not defined over an arbitrary ring.

One has in addition the following equality, which is well-known when $R$ is a field and $\mathcal{M}$ is invertible.
 
\begin{Proposition}\label{adjugate_twice}
	Suppose that the determinant of an $m\times m$ matrix $\mathcal{M}$ is a regular element in the ground ring. Then
	\begin{equation}\label{adjugate_returns_matrix}
		{\rm adj} ({\rm adj}(\mathcal{M}))=(\det\mathcal{M})^{m-2}\, \mathcal{M}.
	\end{equation}
\end{Proposition}
\begin{proof}
	From the basic adjugate relation, since $\det\mathcal{M}$ is regular, $\det {\rm adj}(\mathcal{M})=(\det \mathcal{M})^{m-1}$.
	Applying to ${\rm adj}(\mathcal{M})$ yields, 
	$${\rm adj}(\mathcal{M}) \,{\rm adj} ({\rm adj}(\mathcal{M})) = (\det\mathcal{M})^{m-1} \mathbb{I}_m.$$
	Now multiply both members of the last relation by $\mathcal{M}$, then apply   the basic adjugate relation once more to $\mathcal{M}$, and finally cancel one copy of  $\det\mathcal{M}$ everywhere.
\end{proof}

\subsection{Resiliency of $2\times 2$ minors}

The following result gives a useful relationship between the determinant of a square matrix and the $2$-minors of its adjugate.

\begin{Theorem}\label{adjugate_vs_2minors}
	Let $L$ denote an $m\times m \, (m\geq 2)$ matrix over a  ground ring $R$,  satisfying the following conditions$:$
	\begin{enumerate}
		\item[{\rm (1)}]  $f:=\det L$ is a regular element in $R$.
		\item[{\rm (2)}]  The determinant of some $(m-1)\times (m-1)$ submatrix of $L$ is a regular element over $R/(f)$.
	\end{enumerate}
	Then, for any $2\times 2$ submatrix $T$ of   ${\rm adj}(L)$, $f$ is a factor of $\det T$.
\end{Theorem}
\begin{proof}
	Say, 
	$$T=\left[\begin{matrix}
		t_{i,j} & t_{i,j'}\\
		t_{i',j} & t_{i',j'}
	\end{matrix}\right],
	$$
	and write
	\begin{equation}\label{full_T}
M_{i,i'}=\left[\begin{matrix}
		t_{i,1}&\ldots & t_{i,j} &\ldots & t_{i,j'}& \ldots & t_{i,m}\\
		t_{i',1}&\ldots &t_{i',j} &\ldots & t_{i',j'}& \ldots & t_{i',m}
\end{matrix}\right],
	\end{equation}
for the full $2\times m$ submatrix $M_{i,i'}$ of ${\rm adj}(L)$, with $i$th and $i'$th rows, so that $T$ is a submatrix thereof.

	As guaranteed by assumption (2), let $\tilde{N}$  be an $(m-1)\times (m-1)$ submatrix of $L$ such $\det \tilde{N}\in I_{m-1}(L)$  is regular over $R/(f)$. 
	Let $N$ denote the unique $m\times (m-1)$ submatrix of $L$ containing $\tilde{N}$ as a submatrix.

The following relation
$M_{i,i'}N\equiv \boldsymbol{0}\,({\rm mod\,}f)$
stems from the adjugate formula (\ref{adjugate_relation}), 
where $\boldsymbol{0}$ denotes the null matrix.	
	For even more reason, 
	\begin{equation}\label{sub_adjugate}
		M_{i,i'}\, N\,{\rm adj}(\tilde{N})\equiv \boldsymbol{0}\,({\rm mod\,}f).
	\end{equation}
	Up to a row permutation  we can assume that $\tilde{N}$ is the  upper $(m-1)\times (m-1)$ submatrix of $N.$
	Indeed, for an arbitrary matrix $A\in {\rm GL}(m,R)$ we have $M_{i,i'}\, N=M_{i,i'}\,A^{-1}A\, N$ and $I_{2}(M_{i,i'})=I_{2}(M_{i,i'}A^{-1}).$
	
	Then one has
$$N\,{\rm adj}(\tilde{N})=\left(\begin{matrix}
	\det \tilde{N} &&& \\
	&  \det \tilde{N} &&\\
	& & \ddots & \\
	&&&  \det \tilde{N}\\
	\hline
	p_1 &p_2 &\cdots & p_{m-1}
\end{matrix}\right),
$$
for suitable elements $p_i\in R$, where the empty slots have nullentries.
	
	Therefore, (\ref{full_T}) yields
	$$M_{i,i'}\, N\,{\rm adj}(\tilde{N})=
	\left[\begin{matrix}
		\ldots & t_{i,j}\det \tilde{N}+ t_{i,j'}\gamma & \ldots\\
		\ldots & t_{i',j}\det \tilde{N}+ t_{i',j'}\gamma& \ldots
	\end{matrix}\right],
	$$
	where $\gamma=0$  if $j'\neq m$, and $\gamma=p_j$ if $j'=m.$ 
	
	Then, by \eqref{sub_adjugate}, 
	$$T\left[\begin{matrix} \det \tilde{N}\\ 
		\gamma\end{matrix}\right]
	\equiv \boldsymbol{0}\,({\rm mod\,}f).$$

	In particular,  
	$${\rm adj}(T)\, T\left[\begin{matrix} \det \tilde{N}\\ 
		\gamma\end{matrix}\right]
	\equiv \boldsymbol{0}\,({\rm mod\,}f),$$
	hence,
	$$\left[\begin{matrix} \det T&0\\
		0&\det T
	\end{matrix}\right]
	\left[\begin{matrix}
		\det \tilde{N}\\
		\gamma\end{matrix}\right]
	\equiv \boldsymbol{0}\,({\rm mod\,}f).$$
	Thus,
	$$\det T \det \tilde{N}\in (f)$$ 
	Since $\det \tilde{N}$ is a regular element over $R/(f)$,
	then $\det T\in (f).$
\end{proof}

\subsection{The role of the cofactor}

So far the ground ring $R$ has been fairly general.
Heretofore, it will be typically a polynomial ring over a field $k$, assumed to be graded by ascribing degree $1$ to the indeterminates (standard grading). 
Let $f\in R:=k[x_1,\ldots,x_n]$ be a  homogeneous polynomial in this grading. Introduce  the partial derivatives  $\partial f/\partial x_i, \; 1\leq i\leq n$ of $f$ with respect to the indeterminates of $R$ over $k$.  
Quite often such indeterminates will be double-indexed for convenience.

The next result is customarily quoted as a consequence of \cite{Golberg}. An independent proof appeared in \cite[Proposition 5.3.1]{Maral}. We restate it noting that a certain hypothesis in the latter is superfluous and give a short proof for the reader's convenience.

	Let $\XX=(x_{i,j})_{1\leq i,j\leq m}$ denote an $m\times m$ generic  matrix over the field $k$.
If no confusion arises, we will denote  $k[\XX]:=k[x_{i,j}\,|\,1\leq i,j\leq m]$ the ground polynomial ring on the entries of $\XX$.

\begin{Proposition}\label{GolMar}
	Let $\mathcal{M}$ denote a square matrix over a polynomial ring $R=k[z_1,\ldots,z_n]$  such that  every entry is either 0 or $z_i$ for some $i=1,\ldots,n$. 
	Then, for each $i=1,\ldots, n$, the partial derivative of  $f=\det(\mathcal{M})$ with respect to $z_i$ is the sum
	of the {\rm (}signed{\rm )} cofactors of $z_i$ in all its slots as an entry of $\mathcal{M}$.
\end{Proposition}
\begin{proof}
	More generally, let $\mathcal{N}=(l_{i,j})$ denote an $m\times m$ matrix with linear  entries $l_{i,j}$ in the polynomial ring $R=k[z_1,\ldots,z_n]$. Let $\XX=(x_{i,j})$ stand for the  $m\times m$ generic matrix over $k$, and write $f:=\det \mathcal{N}\subset R$, $g:=\det \XX\subset S:=k[x_{i,j}| 1\leq i\leq j\leq m]$.
	Finally, let $\epsilon: S\surjects R$ denote the surjective $k$-homomorphism given by $x_{i,j}\mapsto l_{i,j}$.
	The ordinary chain rule yields for $1\leq r\leq n$:
	\begin{eqnarray*}
		\partial f/ \partial z_r &=& \sum_{1\leq i,j\leq m} (\partial l_{i,j} / \partial z_r)\, \epsilon ( \partial g/ \partial x_{i,j} )\\
		&=&\sum_{1\leq i,j\leq m} {c}^{i,j}_r \,\epsilon ( \partial g/ \partial x_{i,j} ),
	\end{eqnarray*}
	where ${c}^{i,j}_r $ is the coefficient of $z_r$ in $l_{i,j}$.
	
	Now taking $\mathcal{N}=\mathcal{M}$, the only non-vanishing terms on the right-hand side of the above expression correspond to the entries $l_{i,j}=z_r$.
	On the other hand,  it is well-known or easy to see that  $\partial g/ \partial x_{i,j}$ is the signed cofactor of $x_{i,j}$ in the generic matrix $\XX$. Therefore, whenever an entry $l_{i,j}$ equals some $z_r$, the summand
	${c}^{i,j}_r \epsilon ( \partial g/ \partial x_{i,j} )=\epsilon ( \partial g/ \partial x_{i,j} )$ is the signed cofactor of the entry $z_r$ in slot $(i,j)$ of $\mathcal{M}$.
\end{proof}

Coming back to our first notation, let  $R=k[x_1,\ldots,x_n]$ be a standard graded polynomial ring over the field $k$. 
Let $f\in R$ be a  homogeneous polynomial in this grading, and write $S_f\subset R/(f)$  for the $k$-subalgebra  generated by the residues of the partial derivatives  $\partial f/\partial x_i, \; 1\leq i\leq r$.  
Consider the $k$-algebra map
\begin{equation}\label{dual_variety_map}
	\partial: k[y_1,\ldots,y_n] \surjects S_f, \; y_i\mapsto \partial f/\partial x_i (\bmod f),
\end{equation}
where the $y_i$'s are new indeterminates (often called ``dual'' in the geometric convenience).
From this definition follows immediately that if $f$ is irreducible (respectively, reduced) then $\ker(\mathfrak{\partial})$ is a prime ideal (respectively, a radical ideal).

\begin{Remark}\label{definition_dual_variety}\rm
Though this map is purely algebraic, it plays a vast role in algebraic geometry, residing on the fact that, when $f$ is reduced, $\ker(\partial)$ is identified with the homogeneous defining ideal of the dual variety $V(f)^{\ast}\subset (\mathbb{P}^{n-1})^{\ast}$ to the variety $V(f)\subset \mathbb{P}^{n-1}$ -- in other words,  the homogeneous defining ideal of the image of the Gauss map of  $V(f)$ under the  Pl\"ucker embedding.
\end{Remark}

The following is a well-known device to approach the dual variety of the  generic determinant:

\begin{Proposition}\label{2-minors_vs_dual_variety_generic}
	Let $\XX$ denote an $m\times m$ generic  matrix over the field $k$ and let $f=\det\XX$.
	Let $\YY=(y_{i,j})_{1\leq i, j\leq m}$ denote the matrix in the `dual' variables.
	Then every  $2\times 2$ minor of  $\YY$  vanishes  modulo $(f)$ via the map $y_{i,j}\mapsto \partial f/\partial x_{i,j}$.
\end{Proposition}
\begin{proof}
By Theorem~\ref{adjugate_vs_2minors}, every $2\times 2$ minor of ${\rm adj}(\XX)$ vanishes modulo $(f)$.
By Proposition~\ref{GolMar},  the image of any $2\times 2$ minor of the  matrix $\YY$ by the map  $y_{i,j}\mapsto \partial f/\partial x_{i,j}$   vanishes  modulo $(f)$.
\end{proof}

Proposition~\ref{2-minors_vs_dual_variety_generic} admits convenient extensions to some non-generic matrices encapsulated in the following definition.

\begin{Definition}\label{section_defined}\rm
A {\em section} of the $m\times m$ square generic matrix $\XX$ is a matrix  $L:=(g_{i,j})_{1\leq i, j\leq m}$  whose entries $g_{i,j}$ are forms in the ground polynomial ring $k[\XX]:=k[x_{i,j}| 1\leq i, j\leq m]$ of  $\XX$ (note the continuing abuse of calling both the matrix and the set of its entries by the same symbol $\XX$). 
The section $L$ is moreover said to be {\em linear} if its entries are $1$-forms in $k[\XX]$.
Finally, the section is {\em coordinate sparse linear}  if, for every $i,j$, either $g_{i,j}=x_{i,j}$ or else $g_{i,j}=0$.
\end{Definition}

Consider the generic matrix in the dual variables $\YY$. Read in the dual variables,  $L$ and $L_1$ will be written $L(\YY)$ and $L_1(\YY)$, respectively.

\begin{Theorem} \label{2-minors_vs_dual_variety}
Let there be given a decomposition	${\XX}={\XX}_1\cup {\XX}_2$ as {\rm (}disjoint{\rm )} sets, with the induced decomposition $\YY=\YY_1\cup\YY_2$ in dual variables.
	Let $L=(\ell_{i,j})$ denote a linear section of the $m\times m$ generic matrix $\XX$ such that, either $\ell_{i,j}\in \XX_1$ or else $\ell_{i,j}$ is a linear form with entries in the polynomial ring $k[{\XX}_2]$.
	Assume that $L$ 
	satisfies the following conditions$:$
	\begin{enumerate} 
		\item[{\rm (a)}]  $f:=\det L\neq 0$.
		\item[{\rm (b)}]  Some $(m-1)\times (m-1)$ minor of $L$ is regular over $k[L_1]/(f)$.
		\item[{\rm (c)}] For any two entries $\ell_{i,j},\ell_{i',j'}$ of $L$ contained in $\XX_1$, such that $(i,j)\neq (i',j')$,  one has $\ell_{i,j}\neq \ell_{i',j'}$.
	\end{enumerate}
	Then, up to a change of variables, the ideal of $2\times 2$ minors of $L(\YY)$ with entries in $\YY_2$   is contained in the kernel of the map $\mathfrak{\partial}$ as in {\rm (\ref{dual_variety_map})}, with $\{x_1,\ldots,x_r\}$ {\rm (}respectively, $\{y_1,\ldots,y_r\}${\rm )} replaced by $L_1$  {\rm (}respectively, $L_1(\YY)${\rm )}.
\end{Theorem}
\begin{proof} 
	Since $\mathfrak{\partial}$ commutes with determinants, up to a change of variables any minor in question maps to
	$$\left[\begin{matrix}
		\partial f/ \partial x_{i,j} & \partial f/ \partial x_{i,j'}\\ 
		\partial f/ \partial x_{i',j}& \partial f/ \partial x_{i',j'}
	\end{matrix}\right] ({\rm mod\,}f).$$	
By Proposition~\ref{GolMar} and the assumption in item (c), the entries of the above matrix are cofactors of $L$.
Thus, , the matrix is a submatrix of the adjugate of $L$. Therefore,  the result follows from Theorem~\ref{adjugate_vs_2minors}.
\end{proof}

\begin{Remark}\label{2-minors_vs_dual_variety-sym} \rm
	(1) A similar result can be stated in the case of a symmetric linear section of the generic symmetric matrix.
	
	(2) The proposition is about linear sections because of the use of  Proposition~\ref{GolMar}. One wonders what can be said for sections of fixed degree $\geq 2$.
	
(3) The moral of Theorem~\ref{2-minors_vs_dual_variety}  is that, when $f$ is reduced,  the ideal of $2\times 2$-minors in the proposition  is contained in the homogeneous defining ideal of the dual variety $V(f)^{\ast}$.
	This has an immediate impact to the nature of $V(f)^{\ast}$ since, e.g., it typically implies that the latter is deficient (i.e., not a hypersurface), such as is the case of generic and  symmetric generic matrices.
\end{Remark}

\subsection{Higher minors and cofactor relations}

In this part we focus on relations of the  linear system of the cofactors given by minors of possibly higher order than $2$. 

As before, $\YY=(y_{i,j})_{1\leq i,j\leq m}$ denotes  the $m\times m$ generic matrix over $k$ in the ``dual'' variables, and  $S:=k[y_{i,j}\,|\,1\leq i,j\leq m].$

Let $L$ denote an $m\times m \, (m\geq 2)$ matrix over a $k$-algebra $R$ which is a domain. Consider the $k$-algebra map
\begin{equation}\label{cofactor_map}
	\delta: S \to R, \; y_{i,j}\mapsto \Delta_{i,j},
\end{equation}
where $\Delta_{i,j}=\Delta(L)_{i,j}$ denotes the $(i,j)$-cofactor of $L,$ considered as an element of $R$. 

We are  interested in  the generators of $\ker(\delta)$, and will see that many of these  arise as minors of certain sizes of the matrix $\YY.$ To this end we introduce some further notation.

\begin{Setup}\label{Setup_cofactor}\rm 
As a matter of notation, given an $m\times m$ matrix $M$ with entries in a ring, integers
$1\leq u,v\leq m-1$ and sequences of integers
${\bf i}=\{1\leq i_1< \ldots<i_u\leq m\}\quad  \mbox{and} \quad {\bf j}=\{1\leq j_1< \ldots<j_v\leq m\}$, we denote by $M_{{\bf i},{\bf j}}$ the $u\times v$ submatrix of $M$ with rows (respectively, columns) indexed by the elements of ${\bf i}$ (respectively, of $\bf j$).
Also, for a sequence of integers ${\bf i}\subset \{1,\ldots,m\}$, we denote  by $\mathfrak{c} ({\bf i})$ the complementary set to ${\bf i}$ in $\{1,\ldots,m\}$. 
\end{Setup}

With the above notation, drawing upon the adjugate formula ${\rm adj}(M)\, M=\det M I_{m}$ and matrix block multiplication, yields  the following relation:

\begin{equation}\label{adj_block}
	{\rm adj}(M)_{{\bf i},{\bf j}} M_{{\bf j},\mathfrak{c}({\bf i})}+{\rm adj}(M)_{{\bf i},\mathfrak{c}({\bf j})} M_{\mathfrak{c}({\bf j}),\mathfrak{c}({\bf i})}=\boldsymbol0.
\end{equation}


\begin{Proposition}\label{Image_map_minors}
	Let $L$ and $\delta$ be as in {\rm (\ref{cofactor_map})}.
	Let 
	${\bf i}=\{1\leq i_1< \ldots<i_u\leq m\}$ and  ${\bf j}=\{1\leq j_1< \ldots<j_v\leq m\}$ be  sequences of integers for which the submatrix  $L_{\mathfrak{c}({\bf j}),\mathfrak{c}({\bf i})}$ of $L$ is null. Then
	 the ideal  $I_{v-p+1}(\YY_{{\bf i},{\bf j}})\subset S$  is contained in the kernel of $\delta$.
\end{Proposition}
\begin{proof} Applying \eqref{adj_block} gives
\begin{equation}
	{\rm adj}(L)_{{\bf i},{\bf j}} \, L_{{\bf j},\mathfrak{c}({\bf i})}=\boldsymbol0_{u\times(m-u)}.
\end{equation}
As $p=\rk L_{{\bf j},\mathfrak{c}({\bf i})},$ this implies that $\rk {\rm adj}(L)_{{\bf i},{\bf j}} < v-p+1.$ Hence, $$I_{v-p+1}({\rm adj}(L)_{{\bf i},{\bf j}})=0.$$
Since ${\rm adj}(L)_{{\bf i},{\bf j}}$ is obtained from $\YY_{{\bf i},{\bf j}}$ via the map  (\ref{cofactor_map}), it follows  that $I_{v-p+1}(\YY_{{\bf i},{\bf j}})$ is contained in its kernel, as claimed.
\end{proof}

\begin{Remark}\rm
The proof of the above proposition actually reduces to the case where $L=(a_{i,j})$ is a coordinate sparse linear section of the generic matrix over $k$.

Namely,   consider  the coordinate sparse linear section $G=(g_{i,j})$ such that
$g_{i,j}:= x_{i,j}$,  if $a_{i,j} \neq 0$,  and  $g_{i,j}=0$ otherwise. 
Let $P$ denote the polynomial ring over $k$ generated by the entries of $G$, and consider the $k$-algebra map $P \to R, x_{i,j}\mapsto a_{i,j}$.
Since, by this map, the images of the cofactors of $G$ are cofactors of $L$, then the relations of the former are also relations of the latter.

Thus, examples of such matrices in the coordinate sparse case become relevant.
\end{Remark}

\begin{Example}\rm
	Let $L$ be the following coordinate sparse linear section of the $4\times 4$ generic matrix:
	$$
		\left[\begin{matrix}
		x_{1,1}&x_{1,2}&x_{1,3}&x_{1,4}\\
		x_{2,1}&0&0&x_{2,4}\\
		x_{3,1}&0&x_{3,3}&0\\
		x_{4,1}&x_{4,2}&0&0
	\end{matrix}\right].$$
Consider the following two sequences of integer pairs:
	$${\bf i}_1=\{1,2\},\,\, {\bf i}_2 =\{1,2,3\},\,\, {\bf i}_3=\{1,3\},\,\, {\bf i}_4=\{1,3,4\},\,\, {\bf i}_5=\{1,2,4\},\,\, {\bf i}_6=\{1,4\}$$ 
	$${\bf j}_1=\{1,2,3\},\,\, {\bf j}_2=\{1,2\},\,\, {\bf j}_3=\{1,2,4\},\,\, {\bf j}_4=\{1,4\},\,\, {\bf j}_5=\{1,3\},\,\, {\bf j}_6=\{1,3,4\}.$$
	The  null matrices corresponding to the pairs ${\bf i}_t, {\bf j}_t\, (1\leq t\leq 6)$ are
	$$L_{\mathfrak{c}({\bf j}_1),\mathfrak{c}({\bf i}_1)}=
	\left[\begin{matrix}
		0&0
\end{matrix}\right]
	\,\,
	L_{\mathfrak{c}({\bf j}_2),\mathfrak{c}({\bf i}_2)}=
			\left[\begin{matrix}
		0\\
		0
\end{matrix}\right],
	\,\,
	L_{\mathfrak{c}({\bf j}_3),\mathfrak{c}({\bf i}_3)}=
	\left[\begin{matrix}
		0&0
\end{matrix}\right], 
	$$
	$$
	L_{\mathfrak{c}({\bf j}_4),\mathfrak{c}({\bf i}_4)}=
	\left[\begin{matrix}
		0\\
		0
	\end{matrix}\right],
	\,\,
	L_{\mathfrak{c}({\bf j}_5),\mathfrak{c}({\bf i}_5)}=
	\left[\begin{matrix}
		0\\0
	\end{matrix}\right],
	\,\,
	L_{\mathfrak{c}({\bf j}_6),\mathfrak{c}({\bf i}_6)}=
	\left[\begin{matrix}
		0&0
	\end{matrix}\right] 
	$$
	and the ones we need  their rank are

	$$L_{{\bf j}_1,\mathfrak{c}({\bf i}_1)}=
	\left[\begin{matrix}
		x_{1,3}&x_{1,4}\\
		0&x_{2,4}\\
		x_{3,3}&0
	\end{matrix}\right],\,\,L_{{\bf j}_2,\mathfrak{c}({\bf i}_2)}=
	\left[\begin{matrix}
		x_{1,4}\\
		x_{2,4}
	\end{matrix}\right],\,\, 
	L_{{\bf j}_3,\mathfrak{c}({\bf i}_3)}=
	\left[\begin{matrix}
		x_{1,2}&x_{1,4}\\
		0&x_{2,4}\\
		x_{4,2}&0
	\end{matrix}\right]
	$$
	$$
	L_{{\bf j}_4,\mathfrak{c}({\bf i}_4)}=
	\left[\begin{matrix}
		x_{1,2}\\
		x_{4,2}
	\end{matrix}\right], L_{{\bf j}_5,\mathfrak{c}({\bf i}_5)}=
	\left[\begin{matrix}
		x_{1,3}\\
		x_{3,3}
	\end{matrix}\right],\,\,
	L_{{\bf j}_6,\mathfrak{c}({\bf i}_6)}=
	\left[\begin{matrix}
		x_{1,2}&x_{1,3}\\
		0&x_{3,3}\\
		x_{4,2}&0
	\end{matrix}\right].$$
	Finally, trading dual variables for corresponding cofactors, the following matrices are obtained:
	$$\YY_{{\bf i}_{1},{\bf j}_1}=
	\left[\begin{matrix}
		y_{1,1}&y_{1,2}&y_{1,3}\\
		y_{2,1}&y_{2,2}&y_{2,3}
	\end{matrix}\right],
	\,\, 
	\YY_{{\bf i}_{2},{\bf j}_2}=
	\left[\begin{matrix}
		y_{1,1}&y_{1,2}\\
		y_{2,1}&y_{2,2}\\
		y_{3,1}&y_{3,2}
	\end{matrix}\right],
	\,\,
	\YY_{{\bf i}_{3},{\bf j}_3}=
	\left[\begin{matrix}
		y_{1,1}&y_{1,2}&y_{1,4}\\
		y_{3,1}&y_{3,2}&y_{3,4}
	\end{matrix}\right],
	$$
	$$\YY_{{\bf i}_{4},{\bf j}_4}=
	\left[\begin{matrix}
		y_{1,1}&y_{1,4}\\
		y_{3,1}&y_{3,4}\\
		y_{4,1}&y_{4,4}
	\end{matrix}\right],\,\,
	\YY_{{\bf i}_{5},{\bf j}_5}=
	\left[\begin{matrix}
		y_{1,1}&y_{1,3}\\
		y_{2,1}&y_{2,3}\\
		y_{4,1}&y_{4,3}
	\end{matrix}\right],\,\,
	\YY_{{\bf i}_{6},{\bf j}_6}=
	\left[\begin{matrix}
		y_{1,1}&y_{1,3}&y_{1,4}\\
		y_{4,1}&y_{4,3}&y_{4,4}
	\end{matrix}\right].$$
	By  Proposition~\ref{Image_map_minors}, the ideal
	
	$$(I_2(\YY_{{\bf i}_{1},{\bf j}_1}),I_2(\YY_{{\bf i}_{2},{\bf j}_2}),I_2(\YY_{{\bf i}_{3},{\bf j}_3}),I_2(\YY_{{\bf i}_{4},{\bf j}_4}),I_2(\YY_{{\bf i}_{5},{\bf j}_5}),I_2(\YY_{{\bf i}_{6},{\bf j}_6}))$$
	is contained in the kernel of (\ref{cofactor_map}).
\end{Example}

\subsection{The Grassmann algebra in arbitrary characteristic}

\begin{Definition}\rm
	Let $\XX=(x_{i,j})$ denote an $m\times n (m\leq n)$ generic matrix over an arbitrary field $k$.
	The $k$-subalgebra $G(n,m)$ of the polynomial ring $k[\XX]$, generated by the $m$-minors of $\XX$, is called the {\em Grassmann algebra} of size $(n,m)$.
\end{Definition}

We also record the following well-known result.

\begin{Proposition}\label{dim_Grassman} {\rm (\cite[Corollary 5.12]{BrunsVett})}
 $\dim G(n,m)=m(n-m)+1.$
\end{Proposition}

We give a characteristic-free proof of the next well-known equality in characteristic zero.
The insertion of this argument over here is justified by being largely matrix-theoretic and drawing quite a bit on the adjugate. Besides, we were not able to find the argument  in the literature.

\begin{Theorem}\label{rank_and_dimension}
	Let $\Theta(n,m)$ denote the \index{matrix ! Jacobian} Jacobian matrix of the $m$-minors of $\XX$ as above over an arbitrary field. Then \index{matrix ! Jacobian ! rank} $\rank \Theta(n,m)=\dim G(n,m)$.
\end{Theorem}
\begin{proof}
	Throughout we list the $m$-minors $\Delta_{j_1,\ldots,j_m}$ according to the ordering of their vector of indices in the lexicographic order.
	Likewise, we assume that $\Theta(n,m)$ is written with respect to this ordering.
	
	By Proposition~\ref{dim_Grassman}, we are to show that $\rank \Theta(n,m) =m(n-m)+1$.

	{\sc Claim 1.} $\rank \Theta(n,m) \geq m(n-m)+1$.
	
	Decompose the transpose of $\XX$ in the following way
	
	$$\XX^t=\left[ \begin{matrix} 
		M\\
		L_1\\
		\vdots\\
		L_{n-m}
	\end{matrix}\right]$$
	where $M$ is an $m\times m$ matrix and $L_i$ ($i=1,\ldots,n-m$) are the subsequent row matrices. 
	
	For every index $1\leq i\leq n-m$, denote by $f_{i,1},\ldots,f_{i,m+1}$ the ordered signed $m$-minors of the  $(m+1)\times m$ matrix $\left[\begin{matrix}
	M\\ L_i\end{matrix}\right].$ 
	Here, for any two distinct indices $1\leq i,i'\leq n-m,$ the only common $m$-minor of the respective sets of $m$-minors is $f:=\det M=f_{i,m+1}=f_{i',m+1}$.
	
	\smallskip
	
	{\sc Subclaim.} 
	$L_i {\rm adj}(M)= (-f_{i,1}\cdots -f_{i,m}).$

	In fact, 
		$$\left[\begin{matrix} M\\ L_i\end{matrix} \right] {\rm adj}(M) =
		\left[\begin{matrix}
			f&\cdots&0\\
			\vdots&\ddots&\vdots\\
			0&\cdots&f\\
			\hline
			g_1&\cdots&g_m
		\end{matrix}	\right]
	.$$
	Thus, 
	$$\boldsymbol0=
		[f_{i,1}\cdots f_{i,m+1}]
		\left[\begin{matrix}
			M\\ L_i\end{matrix}\right] {\rm adj}(M)= 
		[f_{i,1}f+g_1f\cdots f_{i,m}f+g_mf]
		$$
	and hence, $g_i=-f_{i,j}$, for  $j=1,\ldots,m,$ thus confirming the subclaim.
	
	Now, let $H$ denote the $(m(n-m)+1)\times (m(n-m)+1)$ Jacobian matrix of  the forms
	$$\{f_{1,1},\ldots,f_{1,m}\}\cup\cdots\cup\{f_{n-m,1},\ldots,f_{n-m,m}\} \cup\{f\}$$ with respect to the variables of the set $\{x_{i,j}\in L_1\}\cup\cdots\cup\{x_{i,j}\in L_{n-m}\}\cup\{x_{m,m}\}.$ 
	
	By the subclaim, for any given  $1\leq i\leq n-m$ and $1\leq j\leq m$, one has:
	$$f_{i,j}=-L_{i}\cdot (\mbox{$j$th column of ${\rm adj}(M)$}).$$
	Since $L_{i}$ and $M$ have no entries in common and neither do $L_i$ and $L_{i'}$ for $i\neq i'$,  we conclude that 
	$H$ has the following shape:
	
	$$\left[\begin{array}{cccc|ccccc}
		-{\rm adj}(M)&\boldsymbol0&\cdots&\boldsymbol0&0\\
		\boldsymbol0&-{\rm adj}(M)&\cdots&\boldsymbol0&0\\
		\vdots&\vdots&\ddots&\vdots&\vdots\\
		\boldsymbol0&\boldsymbol0&\cdots&-{\rm adj}(M)&0\\
		\hline
		\ast&\ast&\cdots&\ast&\Delta(M)_{m,m}
	\end{array}\right]$$
	where $\Delta(M)_{m,m}$ is the cofactor of $M$ with respect to the $(m,m)$th entry, in particular a non-vanishing form.
	Since $ \det ({\rm adj}(M))=(\det M)^{m-1}=f^{m-1}\neq 0$, then $\det H\neq 0$.
	But $H$ is up to elementary row/column operations a submatrix of the Jacobian matrix of $I_{m}(\XX),$ thus wraping up the proof of Claim 1.
	
	\smallskip
	
	{\sc Claim 2.} $\rank \Theta(n,m) \leq m(n-m)+1$.
	
	Write $R:=k[\XX]$ and consider the $R$-homomorphism
	$R^{nm} \stackrel{\Theta(n,m)}{\lar} R^{{n\choose m}}$ defined by $\Theta(n,m)$ in the canonical basis of $R^{nm} $ and $R^{{n\choose m}}$ -- which we call informally the {\em Jacobian map}.
	It suffices to show that the kernel of the Jacobian map  has rank at least $m^2-1$. 
	
	To see  this, for every $1\leq i\leq m$ let $\ell_i$ denote the $i$th column of the transpose of  $\XX$. Now, 
	for every $1\leq i\leq m-1$, consider the following  $(nm)\times m$  matrix
		$$\mathfrak{s}(i)=\left[\begin{matrix}
			\ell_i&\boldsymbol0&\boldsymbol0&\cdots&\boldsymbol0&\cdots&\boldsymbol0&\boldsymbol0\\
			\boldsymbol0&\ell_i&\boldsymbol0&\cdots&\boldsymbol0&\cdots&\boldsymbol0&\boldsymbol0\\
			\vdots&\vdots&\vdots&\ddots&\vdots&\cdots&\vdots&\vdots\\
			\boldsymbol0&\boldsymbol0&\boldsymbol0&\cdots&\ell_i&\cdots&\boldsymbol0&\boldsymbol0\\
			\boldsymbol0&\boldsymbol0&\boldsymbol0&\cdots&\boldsymbol 0&\cdots&\boldsymbol0&\boldsymbol0\\
			\vdots&\vdots&\vdots&\ddots&\vdots&\cdots&\vdots&\vdots\\
			\boldsymbol0&\boldsymbol0&\boldsymbol0&\cdots&\boldsymbol0&\cdots&\ell_i&\boldsymbol0\\
			\boldsymbol0&\boldsymbol0&\boldsymbol0&\cdots&-\ell_m&\cdots&\boldsymbol0&\ell_i
		\end{matrix}\right],
		$$
	where  the column with the entry $\ell_m$ is the $i$th column, and in addition consider the $(nm)\times (m-1)$ matrix
		$$\mathfrak{s}(m)=\left[\begin{matrix}
			\ell_m&\boldsymbol0&\boldsymbol0&\cdots&\boldsymbol0&\cdots&\boldsymbol0\\
			\boldsymbol0&\ell_m&\boldsymbol0&\cdots&\boldsymbol0&\cdots&\boldsymbol0\\
			\vdots&\vdots&\vdots&\ddots&\vdots&\cdots&\vdots\\
			\boldsymbol0&\boldsymbol0&\boldsymbol0&\cdots&\ell_m&\cdots&\boldsymbol0\\
			\boldsymbol0&\boldsymbol0&\boldsymbol0&\cdots&\boldsymbol 0&\cdots&\boldsymbol0\\
			\vdots&\vdots&\vdots&\ddots&\vdots&\cdots&\vdots\\
			\boldsymbol0&\boldsymbol0&\boldsymbol0&\cdots&\boldsymbol0&\cdots&\ell_m\\
			\boldsymbol0&\boldsymbol0&\boldsymbol0&\cdots&\boldsymbol0&\cdots&\boldsymbol0
		\end{matrix}\right].$$
	{\sc Subclaim.} Any among  the $m(m-1)+m-1=m^2-1$ columns of these matrices belongs to the kernel of the Jacobian map.
	
	To prove it,  consider the $m\times m$ submatrix of $\XX$ with column indices $j_1,\ldots,j_m$ 
	$$M_{j_1,\ldots,j_m}=\left[\begin{matrix}
		x_{1,j_1}&\cdots&x_{1,j_m}\\
		\vdots&\ddots&\vdots\\
		x_{m,j_1}&\cdots&x_{m,j_m}
	\end{matrix}\right]$$
	and set $D_{j_1\cdots j_m}:=\det M_{j_1,\ldots,j_m}$.
	For $1\leq r\leq m,$  let $\partial^{(r)}_{j_1,\ldots,j_m}$ denote the vector of the  partial derivatives of  $D_{j_1\cdots j_m}$ with respect to the variables which are the ordered entries of $\ell_r.$ Since $\frac{\partial D_{j_1\cdots j_m}}{\partial x_{i,j_k} }=\Delta(M_{j_1,\ldots,j_m})_{i,k}$ and $\frac{\partial D_{j_1\cdots j_m}}{\partial x_{i,j} }=0$ if $j\notin \{j_1,\ldots,j_m\}$ then, by the adjugate relation 
	$${\rm adj}(M_{j_1,\ldots,j_m})\, M_{j_1,\ldots,j_m}= M_{j_1,\ldots,j_m}\, {\rm adj}(M_{j_1,\ldots,j_m})=D_{j_1\cdots j_m}\, \mathbb{I}_m,$$
	one gets
	$$\partial^{(r)}_{j_1,\ldots,j_m}\, \ell_i=
	\left\{\begin{array}{cc}
		0,&\mbox{if $i\neq r$}\\
		\mbox{$D_{j_1\cdots j_m}$},&\mbox{if $i=r$}
	\end{array}\right.$$
	Hence,
	\begin{equation}\label{syzthetamn1}
		\partial^{(r)}_{j_1,\ldots,j_m}\, \ell_i= 0,\quad \mbox{if $i\neq r$}
	\end{equation}
	and
	\begin{equation} \label{syzthetamn2}
		\partial^{(i)}_{j_1,\ldots,j_m}\, \ell_{i}-\partial^{(m)}_{j_1,\ldots,j_m}\cdot \ell_{m}=0	\quad\mbox{for $1\leq i\leq m-1$}.
	\end{equation}
	
	The relations \eqref{syzthetamn1} and  \eqref{syzthetamn2} prove the stated subclaim.

	To wrap up the proof of Claim 2, let  $\mathbb{K}(n,m)$ stand for the concatenation of the  matrices  $\mathfrak{s}(i)$ ($1\leq i\leq m-1$) and $\mathfrak{s}(m)$, whose columns as just seen belong to the kernel of the Jacobian map.
	We now show that $I_{m^2-1}(\mathbb{K}(n,m))$ has positive height, which implies in particular that $\mathbb{K}(n,m)$ has rank $m^2-1$.
	For this, note that up to column permutation we can write the $nm\times (m^2-1)$ matrix $\mathbb{K}(n,m)$ in the form
	$$\left[\begin{matrix}
		\XX^t&\\
		&\XX^t\\
		&&\ddots\\
		&&&\XX^t\\
		\ast&\ast&\cdots&\ast&\overline{\XX}^t
	\end{matrix}\right]$$
	where the empty slots are null entries and $\overline{\XX}$ denotes the $(m-1)\times n$ submatrix obtained from $\XX$ omitting the $m$th row. Thus, for every $m\times m$ submatrix $A$ of $\XX$ and every $(m-1)\times(m-1)$ submatrix $B$ of $\overline{\XX}^t$ we have the following $(m^2-1)\times(m^2-1)$ submatrix of $\mathbb{K}(n.m)$ 
	$$\left[\begin{matrix}
		A^t&\\
		&A^t\\
		&&\ddots\\
		&&&A^t\\
		\ast&\ast&\cdots&\ast&B^t
	\end{matrix}\right]$$
	whose determinant is $\det A^{m-1}\det B.$ Thus, a minimal prime of $I_{m^2-1}(\mathbb{K}(n,m))$ contains $\det A\det B$ for all choices of $A$ and $B$. It follows that such a prime
	contains either $I_{m}(\XX)$ or $I_{m-1}(\overline{\XX})$, in particular, it has height $\geq n-m+1$ or $\geq n-m+2.$ Thus, $\Ht I_{m^2-1}(\mathbb{K}(n,m))\geq 1.$
\end{proof}

\subsection{Linear syzygies I: Gorenstein ladder}

For the notion of syzygies and minimal graded free resolutions see \cite[Section 3.5 and Section 7.1]{SimisBook}.
In this and next subsection we exploit quite strongly the idea of collecting many syzygies via the adjugate.

Consider  the following kind of one-sided ladder of the $m\times m$ generic matrix $\XX$:

\begin{center}
	\begin{tikzpicture}
		\draw (1.9,1.2) node {i};
		\draw (2.9,2.2) node {i};
		\draw (3.9,3.2) node {i};
		\draw (2,-0.2) node {o};
		\draw (2.8,1.2) node{o};
		\draw (3.8,2.2) node{o};
		\draw (5.2, 3) node {o};
		\draw (0,0)--(0,5);
		\draw (0,0)--(2,0);
		\draw (5,3)--(5,5);
		\draw (0,5)--(5,5);
		\draw (2,0)--(2,1);
		\draw (2,1)--(3,1);
		\draw (3,1)--(3,2);
		\draw (3,2)--(4,2);
		\draw (4,2)--(4,3);
		\draw (4,3)--(5,3);
	\end{tikzpicture}
\end{center}
In the above picture, the number of inside corners is three. 
One is interested in the ideals of minors confined to the (left side of the) ladder.
If the ladder is denoted $\mathcal{L}$, then such ideals are denoted $I_t(\mathcal{L})$, for various values of $t$.

We will focus in the case where there is only one inside corner.
It is well known that the ideal of submaximal minors of this ladder is a codimension three Gorenstein ideal (\cite[Theorem 4.9 (c)]{Conca1995}), hence it is generated by the maximal Pfaffians of a certain alternating matrix.
We give next the shape of such a matrix as based on the adjugate ${\rm adj}(\XX)$.

\begin{Proposition}
Let $\mathcal{L}$ denote the above ladder of the $m\times m$ generic matrix $\XX=(x_{i,j})$, with one single inside corner.
Then $I_{m-1}(\mathcal{L})$ is the ideal of maximal  Pfaffians of the $(2m-1)\times (2m-1)$ alternating matrix
\begin{equation*}
	\Phi_m=	\left(\begin{array}{cccccccccc}
		0&\cdots&0&x_{m,1}&x_{m-1,1}&\cdots&x_{1,1}\\
		\vdots&&\vdots&\vdots&\vdots&&\vdots\\
		0&\cdots&0&x_{m,m-1}&x_{m-1,m-1}&\cdots&x_{1,m-1}\\
		-x_{m,1}&\cdots& -x_{m,m-1}&0&-x_{m-1,m}&\cdots&-x_{1,m}\\
		-x_{m-1,1}&\cdots&-x_{m-1,m-1}&x_{m-1,m}&0&\cdots&0\\
		\vdots&&\vdots&\vdots&\vdots&&\vdots\\
		-x_{1,1}&\cdots& -x_{1,m-1}&x_{1,m}&0&\cdots&0
	\end{array}\right).
\end{equation*}	
\end{Proposition}
\begin{proof}
	Let  $\Delta_{i,j}$ denote the $(i,j)$th entry of the adjugate matrix ${\rm adj}(\XX)$. \index{matrix ! adjugate} Then the canonical generators of $I_{m-1}(\mathcal{L})$ are identified with the $2m-1$ cofactors
	$$\Delta_{i,m} (1\leq i\leq m),  \Delta_{m,j} (1\leq j\leq m-1).$$ 
	Among the relations obtained from the adjugate formula   (\ref{adjugate_relation})    of $\XX$,   we single out the following ones  
	$$\sum_{j=1}^{m}x_{i,j}\Delta_{j,m}=0 \quad (1\leq i\leq m-1),$$
	$$\sum_{j=1}^{m-1}x_{m,j}\Delta_{j,m}-\sum_{i=1}^{m-1}x_{i,m}\Delta_{m,i}=0,$$
	$$\sum_{i=1}^{m}x_{i,j}\Delta_{m,i}=0 \quad (1\leq j\leq m-1).$$
	
	Rewriting the canonical generators as
	\begin{equation}\label{funny_gens}
		\{\Delta_{1,m},\Delta_{2,m},\cdots,\Delta_{m-1,m},-\Delta_{m,m},-\Delta_{m,m-1},\cdots,-\Delta_{m,2},-\Delta_{m,1}\},
	\end{equation} 
	these relations give the columns of $\Phi_m$, hence the latter becomes a  matrix of syzygies of the ideal $I_{m-1}(\mathcal{L})$.
	
	In order to show that  is the full matrix of syzygies of $I_{m-1}(\mathcal{L})$, one may resort to the following complex of free $R$-modules:
	\begin{equation}\label{resolution_Gor_cod3}
		0\to R(-(2m-1))\stackrel{\boldsymbol\Delta^t}\longrightarrow R(-(m))^{2m-1}\stackrel{\Phi}\longrightarrow R(-(m-1))^{2m-1}\stackrel{\boldsymbol\Delta}\longrightarrow R
	\end{equation}
	where $\Phi$,  $\boldsymbol{\Delta}$  and $\boldsymbol{\Delta}^t$  respectively denote the maps, in the canonical bases, associated respectively to $\Phi_m$, the $1\times (2m-1)$ matrix whose entries are the cofactors in (\ref{funny_gens}), and its transpose. 
	We now apply the \index{Buchsbaum, D. ! Buchsbaum--Eisenbud criterion} 
	Buchsbaum--Eisenbud criterion (\cite[Section 1, Theorem]{BE_Criterion1973}).
	
	Obviously, the ranks add up correctly.
	In addition, 
	$\grade  I_1(\boldsymbol{\Delta})= \grade  I_{m-1}(\mathcal{L})= 3$ -- this value alone can be derived as a particular case of the formula in \cite[Theorem 5.3 and remarks on p. 243)]{Stanley1977}. Finally,
	$\grade  I_{2m-2}(\Phi_m)\geq 2$, as is easily checked.
	Thus, the above complex is exact, hence $\Phi_m$ is indeed the full syzygy matrix -- as a bonus, this proves again that $I_{m-1}(\mathcal{L})$ is Gorenstein.
\end{proof}

\subsection{Linear syzygies II: sparse symmetric section}\label{lin_syz_II}

For the encompassing theory of birational maps we refer to \cite[Section 3.2]{GradedBook}, where the notion of a homaloidal form can be found in  Subsection 3.2.5.

Among the symmetric sparse linear sections of the symmetric matrix, the following coordinate sparse one has particular properties:
\begin{equation}\label{hollow_symmetric_size_m}
	\mathfrak{H}=\mathfrak{H}_m:=\left(\begin{array}{ccccc}
		x_{1,1}&x_{1,2}&\cdots &x_{1,m-1}&x_{1,m}\\
		x_{1,2}&0&\cdots&0&x_{2,m}\\
		x_{1,3}&0&\cdots&x_{3,m-1}&0\\
		\vdots & \vdots & \iddots& \vdots &\vdots\\
		x_{1,m-1}&0&x_{3,m-1}&0&0\\
		x_{1,m}&x_{2,m}&0&0&0
	\end{array}\right).
\end{equation}
Set $\mathfrak{f}:=\det \mathfrak{H}$.

\begin{Proposition}\label{determinant_is_homaloidal} {\rm (char$(k)\neq 2$)}
$\mathfrak{f}$ is homaloidal.
\end{Proposition}
\begin{proof}
Write $\mathfrak{f}_{i,j}:=\partial \mathfrak{f}/\partial x_{i,j}$,  and $\Delta_{i,j}$ for the $(i,j)$th cofactor of $\mathfrak{H}$. 
Since $\mathfrak{H}$ is a symmetric matrix, $\Delta_{i,j}=\Delta_{j,i}$ for any $(i,j)$. By Proposition~\ref{GolMar}, $\mathfrak{f}_{i,j}=\Delta_{i,j}$ possibly up to multiplication by $2$ (recall that char$(k)\neq 2$). Therefore, the corresponding  fields of fractions coincide. Thus, it is equivalent to show that the rational map defined by the set $\boldsymbol\Delta=\{\Delta_{i,j}\,|\, (i,j)\in\{(1,1), (1,2), \ldots, (1,m),(i,m+2-i)\}$, for $2\leq i\leq \lfloor m/2\rfloor+1$, is birational.
	
	Now, the adjugate matrix ${\rm adj}(\mathfrak{H}_m)$ has the following shape
	\begin{equation}\label{adjugate_of_hollow_symmetric}
		{\rm adj}(\mathfrak{H}_m)=\left(\begin{array}{cccccc}
			\Delta_{1,1}&\Delta_{1,2}&\cdots&\Delta_{1,m-1}&\Delta_{1,m}\\
			\Delta_{1,2}&\ast&\cdots&\ast&\Delta_{2,m}\\
			\Delta_{1,3}&\ast&\cdots&\Delta_{3,m-1}&\ast\\
			\vdots&\vdots&\iddots&\vdots&\vdots\\
			\Delta_{1,m}&\Delta_{2,m}&\cdots&\ast&\ast
		\end{array}\right),
	\end{equation}
	where $\ast$ denotes some  unspecified entry.
	From  the  \index{matrix ! adjugate} adjugate formula (\ref{adjugate_relation}), one gets the following  linear  relations among members of the set $\boldsymbol\Delta$:
	\begin{eqnarray*}
		x_{1j}\Delta_{1,1}+  x_{j,m+2-j}\Delta_{1,m+2-j}=0,\quad 2\leq j\leq t\\
		x_{1j}\Delta_{1,1}+  x_{m+2-j,j}\Delta_{1,m+2-j}=0,\quad t+1\leq j\leq m,
	\end{eqnarray*}
	as a result of the zero entries in $\mathfrak{f}\, \mathbb{I}_{m\times m}$, and 
	\begin{equation*}
		\left(\sum_{j=1}^{m} x_{1,j}\Delta_{1,j}\right)- x_{1,i}\Delta_{1,i}-  x_{i,m+2-i}\Delta_{i,m+2-i}=0,\quad 2\leq i\leq t,
	\end{equation*}
	as the expression of equating two different terms that repeat the same  entry along the diagonal of $\mathfrak{f}\, \mathbb{I}_{m\times m}$.
	
	These relations give rise to the following $(m+t-1)\times (m+t-2)$ submatrix of the syzygy matrix of the set $\boldsymbol\Delta:$
	{\small \begin{equation*}
			\mathcal{Z}=
			\left(\begin{array}{ccccccc|cccccccc}
				x_{1,2}&  x_{1,3}&\cdots& x_{1,t}& x_{1,t+1}&\cdots& x_{1,m}& x_{1,1}&\cdots& x_{1,1}\\
				0&0&\cdots&0&0&\cdots& x_{2,m}&0&\cdots& x_{1,2}\\
				\vdots&\vdots&&\vdots&\vdots&&\vdots&\vdots&&\vdots\\
				0&0&\cdots&0& x_{t,m+2-t}&\cdots&0& x_{1,t}&\cdots&0\\
				0&0&\cdots& x_{t,m+2-t}&0&\cdots&0& x_{1,t+1}&\cdots& x_{1,t+1}\\
				\vdots&\vdots&&\vdots&\vdots&&\vdots&\vdots&&\vdots\\
				0& x_{3,m-1}&\cdots&0&0&\cdots&0& x_{1,m-1}&\cdots& x_{1,m-1}\\
				x_{2,m}&0&\cdots&0&0&\cdots&0& x_{1,m}&\cdots& x_{1,m}\\
				\hline
				0&0&\cdots&0&0&\cdots&0&- x_{2,m}&\cdots&0\\
				\vdots&\vdots&&\vdots&\vdots&&\vdots&\vdots&&\vdots\\
				0&0&\cdots&0&0&\cdots&0&0&\cdots&- x_{t,m+2-t}\\
			\end{array}\right)	
		\end{equation*}
	}
	
	Though it can be shown that $\mathcal{Z}$ has full rank, we don't as yet know whether the set $\boldsymbol\Delta$ is algebraically independent over $k$, so that we could conclude by the general theory of \cite{AHA}.
	We will instead argue directly upon the submatrix of the Jacobian dual matrix induced by $\mathcal{Z}$.
	Namely, introducing `dual' variables 
	$$\YY= \{y_{1,j} \,(1\leq j\leq m ), y_{i,m+2-i}\,(2\leq i\leq t)\}$$
	the transpose of the following matrix is a submatrix of the Jacobian dual matrix  of the set $\boldsymbol\Delta$
	{\small 
		\begin{equation}
			B=	\left(\begin{array}{cccccccc|ccccccccccc}
				0&0&\cdots&0&0&\cdots&0&0& y_{1,1}&\cdots& y_{1,1}\\
				y_{1,1}&0&\cdots&0&0&\cdots&0&0&0&\cdots& y_{1,2}\\
				0& y_{1,1}&\cdots&0&0&\cdots&0&0& y_{1,3}&\cdots& y_{1,3}\\
				\vdots&\vdots&&\vdots&\vdots&&\vdots&\vdots&\vdots&&\vdots\\
				0&0&\cdots& y_{1,1}&0&\cdots&0&0& y_{1,t}&\cdots&0\\
				0&0&\cdots&0& y_{1,1}&\cdots&0&0& y_{1,t+1}&\cdots& y_{1,t+1}\\
				\vdots&\vdots&&\vdots&\vdots&&\vdots&\vdots&\vdots&&\vdots\\
				0&0&&0&0&\cdots& y_{1,1}&0& y_{1,m-1}&\cdots& y_{1,m-1}\\
				0&0&&0&0&\cdots&0& y_{1,1}& y_{1,m}&\cdots& y_{1,m}\\
				\hline
				y_{1,m}&0&\cdots&0&0&\cdots&0& y_{1,2}&- y_{2,m}&\cdots&0\\
				0& y_{1,m-1}&\cdots&0&0&\cdots& y_{1,3}&0&0&\cdots&0\\
				\vdots&\vdots&&\vdots&\vdots&&\vdots&\vdots&\vdots&&\vdots\\
				0&0&\cdots& y_{1,t+1}& y_{1,t}&\cdots&0&0&0&\cdots&- y_{t,m+2-t}\\
			\end{array}\right)
		\end{equation}
	}
	
	Let $B_1$ be the $(m+t-2)\times (m+t-2)$ submatrix of $B$  obtained by omitting the  first row. By \cite[Theorem 2.18]{GradedBook}, it suffices to show that $\det B_1$ does not vanish when the entries of $B_1$ are evaluated via $y_{i,j}\mapsto \Delta_{i,j}$. We have that 
	$$\det B_1=y_{1,1}^{m-1}\prod_{i=2}^{t} y_{i,m+2-i}+T$$
	where $T$  is a polynomial in $k[\yy]$ that does not depend on $y_{1,1}.$
	Evaluating as indicated gives the expression
	\begin{equation*}\label{detB1_ev}
		\Delta_{1,1}^{m-1}\prod_{i=2}^{t} \Delta_{i,m+2-i}+T(\Delta_{1,2},\ldots,\Delta_{1,m},\Delta_{2,m},\ldots,\Delta_{t,m+2-t})
	\end{equation*}
	which is non-vanishing because, e.g.,  while its second summand  belongs to the ideal  $(x_{1,2},\ldots,x_{1,m})^{m+t-1}$, the first does not.
	\end{proof}

\section{Square syzygy matrices of codimension two ideals}\label{Sec2}

What are the `natural' ideals related to an $m\times m$ matrix $\mathcal{M}$ over a polynomial ring $k[x_1,\ldots, x_n]$ ($k$ a field)?
The first such examples are the Fitting ideals of $\mathcal{M}$. 
A good window for discussion is the codimension of these ideals. For example, if we require that this is at most $2$, then natural ideals are $\det \mathcal{M}$ and $ I_{m-1}(\mathcal{N})$, for any submatrix $\mathcal{N}$ of $\mathcal{M}$ upon omitting one row or column.
These tend to be well-structured, for reasonably structured matrices.


Since the case of $\det \mathcal{M}$ itself sits on the theory of hypersurface rings,  we aim in this section at  square matrices closely related to ideals of codimension two, say, in terms of the corresponding syzygies of the ideal in question.
The main bulk of work has so far focused in codimension two ideals which happen to be perfect.
We will instead focus on the case of a non-perfect ideal with a square syzygy matrix -- so to say, the next non Cohen-Macaulay case.

\subsection{Sub-Hankel syzygies}
Consider the following sparse symmetric matrix

\begin{equation*}\label{subHankel_matrix}
	\mathcal{H}=\left(
	\begin{matrix}
		x_0&x_1&x_2&x_3&\ldots &x_{m-2} & x_{m-1}\\
		x_1&x_2& x_3&x_4 &\ldots &x_{m-1}& x_{m}\\
		x_2&x_3& x_4&x_5&\ldots  &x_{m}& 0\\
		\vdots &\vdots & \vdots&\vdots  &\ldots &\vdots &\vdots \\
		x_{m-2}&x_{m-1}& x_m &0&\ldots  & 0 &0\\
		x_{m-1}&x_{m}& 0 &0& \ldots  & 0 &0
	\end{matrix}
	\right)
\end{equation*}
over a polynomial ring $R:=k[x_0, \ldots, x_m]$ ($k$ a field).
It has been dubbed a {\it sub--Hankel matrix} of order $m$ (see \cite{CRS}). 
Let $f:=\det \mathcal{H}\subset R$.

Consider the partial derivatives $f_{x_i:=}\partial f/\partial x_i, 0\leq i\leq m$ of $f$ and let $J=J_f\subset R$ denote the gradient ideal of $f$.
Next,  a collection of the main properties of $J$ which first appeared in \cite{CRS} (see also \cite{Maral}).

\begin{Theorem}
	With the above notation, the following holds$:$
	\begin{enumerate}
		\item[{\rm (a)}] For $0\leq i\leq m-1$, one has		
		\begin{equation}\label{g.c.d.s}
			\{f_0,\ldots,f_i\} \subset k[x_{m-i},\ldots,x_{m}].
		\end{equation}
		and the g.c.d. of these partial derivatives is $x_m^{m-i-1}$.
		\item[{\rm (b)}] The ideal $J$ has maximal linear rank.
		More precisely, the following linear syzygies of $J$ generate an $R$-submodule of rank $m$:
		\begin{equation*}\label{perfect_linear_relation} 
			\sum_{l=0}^{m-2} (m-l-1)\,
			x_l\, f_l-x_m\,f_m=0,
		\end{equation*}
		and
		\begin{equation*}\label{basic_linear_relation}
			2ix_{m-i}f_0+\cdots+(2i-l)x_{m-i+l}f_l+\cdots+ix_mf_i=0, {\rm (}1\leq i\leq m-1{\rm )},
		\end{equation*}
		for {\rm (}$1\leq i\leq m-1${\rm )}.
	\end{enumerate}
\end{Theorem}
Since $f\in (x_{m-1},x_m)$, then $J\subset (x_{m-1},x_m)$ as well.
In addition, $\Delta_{1,1}=x_m^{m-1}\in J$ as is easily verified.
Therefore, $J$ has codimension two.
We now display its homological behavior.

\begin{Proposition}\label{resolution of J}
	The minimal graded resolution of $R/J$ has the form
	\begin{equation*}
		0\rar R(-(2m-1))\stackrel{\psi^t}{\lar} R(-m)^m\oplus
		R(-2(m-1))\stackrel{\phi}{\lar} R^{m+1}(-(m-1))\rar R,
	\end{equation*}
	where
	\begin{equation*}\phi=
		\left(\begin{matrix}
			(m-1)x_0 & (2m-2)x_1 &  (2m-4)x_2 &\ldots &2x_{m-1} & 0\\
			(m-2)x_1 &  (2m-3) x_2&  (2m-5) x_3 &\ldots &x_m & 0 \\
			(m-3)x_2 &   (2m-4)x_3&   (2m-6)x_4 & \ldots & 0 & 0\\
			\vdots & \vdots & \vdots & \vdots  & \vdots & 0\\
			x_{m-2} & mx_{m-1} &  (m-2)x_m & \ldots &0 & 0\\
			0& (m-1)x_m & 0 & \ldots & 0  & -f_{x_m}\\
			-x_m & 0 & 0 & \ldots & 0 & f_{x_{m-1}}
		\end{matrix}\right)
	\end{equation*}
	and 
	$$\psi= (\psi_1 \cdots \psi_m), \: \psi_i=\Delta_i/\mathcal{D},$$  
	with $\Delta_i$ denoting the $i$th $m$-minor of an $m\times (m+1)$ submatrix of $\phi$ of rank $m$, and $\mathcal{D}=\gcd (\Delta_0, \ldots,\Delta_m)$.
\end{Proposition}
Here are some comments on this example.

The $3$-syzygy format of this resolution was established in \cite{MarSim2016} (also \cite{Maral}).
The details of $\phi$ were known to \cite{CRS}, while those of $\psi^t$ are common knowledge derived from the generalities of free resolutions plus the fact that $J$ has codimension two -- by taking an explicit submatrix the details of $\psi^t$ can be listed down in terms of the syzygies of subideals of $J$.

One can also show that $J$ is an ideal of linear type, but this is way harder (again see \cite{MarSim2016}).
We note en passant that \cite[Theorem 3.1]{dual1993} is not applicable: although condition (b) there is satisfied, one of the standing assumptions fails as, for $m\geq 3$,  $J$ is  not perfect locally on the punctured spectrum because $I_1(\psi^t)$ has codimension three. 
But, see next subsection for an encompassing class of ideals $I$ with the standing assumptions of  \cite[Theorem 3.1]{dual1993}, but with ${\rm indeg} \,\syz(I)=\dim R$.

Finally, it should be pointed out that this example is a linear case of the more recent (determinantal) nearly-free hypersurface according to \cite[Definition 2.6]{Dimca_et_al2021}.

\subsection{Inversion factors of a general hypersurface embedding}

{\sc Setup.} Let $L$ denote an $m\times n\, (m\geq n)$ matrix whose entries are linear forms in a polynomial ring $R:=k[x_1,\ldots,x_d] \,(d\geq 2)$.

\begin{enumerate} 
	\item[{\rm (a)}]
	Let $\mathfrak{M}:\pp^{d-1}\dasharrow \pp^{{m \choose n}-1}$ denote the rational map defined by the linear span of the $n\times n$ minors of $L$.
	The determinantal ideal $I_n(L)$ (or its associated determinantal variety) is {\em homaloidal} if $\mathfrak{M}$ is birational onto its (closed) image.
	By abuse, we  then say that $L$ is {\em homaloidal-like}.
	
	Note that this property forces the inequality ${m \choose n}\geq d$, which is no restriction for a square matrix, but for an $m\times (m-1)$ matrix it requires $m\geq d$.
	In this regard, for $d=3$, every $m\times (m-1)$ linear matrix $L$ is homaloidal-like provided $\codim I_{m-1}(L)=2$ and $\codim I_1(L)=3$ (\cite{linpres2018}).
	For higher $d$, one could first consider the case of `general' linear matrices in the sense of the next two items.
	
	\item[{\rm (b)}] Set $\pp^{d-1}={\rm Proj}_k(R)$. Recall that the parameter space of linear forms in $R$ is the dual space ${\pp^{d-1}}^*$ in dual coordinates $y_0, \ldots,y_d$.
	It comes up with the tautological map $\tau: R_1\lar {\pp^{d-1}}^*$ which associates to a linear form  the point of ${\pp^{d-1}}^*$ whose coordinates are the  coefficients of $\ell$ in $k$. 
	One says that the {\em general} linear form $\ell\in R$ has a property $\mathcal{P}$ if the set of points $\tau(\ell)$ such that $\ell$ has this property contains a dense open set of ${\pp^{d-1}}^*${\rm (}i.e.,  avoids a proper closed set of ${\pp^{d-1}}^*${\rm )}.
	There is also a notion for a finite set of linear forms $\{\ell_1,\ldots,\ell_r\}\subset R$. For this, we consider the product space
	$${\pp^{d-1}}^*\times_{k}\cdots \times_{k} {\pp^{d-1}}^*$$
	along with the  tautological map
	$$\tau_r:=\underbrace{\tau\times \cdots \times \tau}_r: R_1\times_{k}\cdots \times_{k} R_1\lar {\pp^{d-1}}^*\times_{k}\cdots \times_{k} {\pp^{d-1}}^*.$$
	
	Then the {\em general} set of $r$ linear forms  in $R$ has a certain property  if the set of points $\tau_r(\ell_1,\ldots,\ell_r)$ such that $\{\ell_1,\ldots,\ell_r\}$ has this property contains a dense open set.
	Note that such an open set is complementary to a proper closed set defined by multi-homogeneous forms.
	\item[{\rm (c)}] Finally, we will say that the {\em general linear matrix} of order $m\times n$ has a certain property if the set $\mathcal{L}$ of $m\times n$ linear matrices having this property is such that the set of parameter points corresponding to its nonzero entries contains a dense open set.
\end{enumerate}

The following result is the base of the family discussed in the present part:
\begin{Proposition}{\rm (\cite[Proposition 2.9 (a)]{SymbPowBir2014})}
	Let $R=k[x_1,\ldots,x_d]\, (d\geq 3)$. For any integer $m\geq d$, the general linear matrix of order $m\times (m-1)$ is homaloidal-like.
\end{Proposition}
Next assume that $m=d+1$.
(Note that, in this case, the (closed) image of the birational map in question is a sufficiently general irreducible hypersurface, justifying the subsection title.)

Let $\yy=\{y_1,\ldots,y_{d+1}\}$ denote the target coordinates of the map defined by the $d\times d$ minors.
Write $(\yy) L= (x_1 \dots x_d) B$, where $B$ is a $d\times d$ matrix over $k[\yy]$ -- the transpose of $B$ is called a {\em Jacobian dual} matrix to $L$ (\cite{dual1993}).

By the birational theory attached to the Jacobian dual, if $L$ is homaloidal-like, the $(d-1)\times (d-1)$ minors of any $(d-1)\times d$ submatrix of $B$ of rank $d-1$ define the inverse map (\cite{AHA}) -- called a {\em representative} of the inverse map.
Since the property of maximal rank is open in the parameter space, if $L$ is general then any $(d-1)\times d$ submatrix of $B$ will do.
This gives $d$ representatives of the inverse map.
Thus, when $L$ is general, we call these representatives {\em basic representatives}.

Now, each such a representative gives rise to a so-called  inversion factor (\cite[Section 4.2.1]{GradedBook}).
Namely, if $\{g_1, \dots ,g_d\}\subset k[\yy]$ is such a representative, there is a unique form $D\in R$ such that $g_i(\boldsymbol{\Delta})=Dx_i$, for every $1\leq i\leq d$ -- here $\boldsymbol{\Delta}$ denotes the set of $d\times d$ minors of $L$.
This form $D$ is called the {\em inversion factor} associated to the given representative.

We are ready for the main case of this subsection.

\begin{Theorem}\label{inversion_factors_resolution}
	With the above notation, the general linear matrix of order $(d+1)\times d$ satisfies the property that the set $\{D_1, \ldots, D_d\}\subset R$ of inversion factors associated to the basic representatives generate a codimension two ideal with minimal free resolution
	\begin{equation}\label{resolucaodosDs}
		0\rar\; R(-d^2)\;\stackrel{\xx^t}{\rar}\; R(-(d^2-1))^d\;\stackrel{\Psi}{\rar}R(-(d(d-1)-1))^{d}\;\rar R,
	\end{equation}
	where $\Psi$ denotes the resulting matrix of evaluating every entry of the Jacobian dual matrix to $L$ at
	the {\rm (}signed{\rm )} minors, while $\xx^t$ stands for the transpose of the vector $\xx$.
\end{Theorem}
\begin{proof}
	We first analyze $\Psi$.
	Although its structure is fairly simple, its entries are sort of wild. However, the $(d-1)$-minors are pretty manageable because they can be thought of as cofactors. In fact, letting $\delta_{i,j}$ denote a typical minor of any $(d-1)\times d$ submatrix of the Jacobian dual matrix, we know it is a coordinate of a basic inverse representative. Evaluating at $\bf \Delta$, by defintion of inverse factor, one has $x_i D_j=\delta_{i,j}(\bf \Delta)$.
	It follows that the adjugate matrix of $\Psi$ is
	\begin{equation}\label{cofatores}
		{\rm adj}(\Psi)=\left(\begin{array}{cccccccccccc}
			x_1D_1&x_1D_2&\ldots&x_1D_d\\
			x_2D_1&x_2D_2&\ldots&x_2D_d\\
			\vdots&\vdots&\ldots&\vdots\\
			x_dD_1&x_dD_2&\ldots&x_dD_d
		\end{array}\right)
	\end{equation}
	Since $\Psi$ has rank $d-1$, the adjugate relation gives
	\begin{equation}\label{idcofatores1}
		{\rm adj}(\Psi)\,\Psi=\bf{0}
	\end{equation}
	and
	\begin{equation}\label{idcofatores2}
		\Psi\,{\rm adj}(\Psi)={\bf 0}
	\end{equation}
	But, (\ref{idcofatores1}) implies that $\Psi$ is a matrix of syzygies of $\bf D$, while (\ref{idcofatores2})
	gives that $\XX^t$ is a second syzygy thereof.
	This shows that one has indeed a complex.
	
	To finish we check the required Fitting codimension by the Buchsbaum--Eisenbud acyclicity criterion.
	The verification at the tail of the complex is immediate, while at the middle one has
	$$I_{d-1}(\Psi)=I_1({\rm adj}(\Psi))=(\xx)(D_1,\ldots, D_d).$$
	Thus, we only need to prove that $(D_1,\ldots, D_d)$ has codimension $\geq 2$ -- and hence,  $\codim (D_1,\ldots, D_d)= 2$ (otherwise the resolution would imply that $(D_1,\ldots, D_d)$ is Cohen--Macaulay of type one, which means it is a Gorenstein ideal; the latter being self-dual would imply that $(D_1,\ldots, D_d)=(x_1,\ldots, x_n)$, that is, $\Psi$ would be the Koszul matrix of $(x_1,\ldots, x_n)$, a contradiction for degree reasons (at least!)).
	
	The reader is referred to \cite[The proof of Theorem 2.19 (i)]{SymbPowBir2014} for a direct proof of the equality $\codim (D_1,\ldots, D_d)= 2$.
\end{proof}

\begin{Conjecture}\rm 
	The symmetric algebra of the ideal $(D_1,\ldots, D_d)$ is Cohen--Macaulay of dimension $d+1$ and type two.
\end{Conjecture}
We note that, according to \cite[Theorem 3.1]{dual1993}, it suffices to show that $(D_1,\ldots, D_d)$ satisfies  condition $F_0$ for the codimension growth of the Fitting ideals of the matrix $\Psi$.
\begin{Remark}\rm
	For a $4\times 3$ general matrix in dimension $d=3$ we have verified this by computer assistance, using random coefficients. It is a tricky challenge to write such explicit examples without random coefficients because one has to carry around the assumption of  general forms, forcing the required dense open set to get smaller and smaller. 
\end{Remark}

\subsection{$3$-syzygy matrices in three variables}


In this part we discuss some  properties of a three generated ideal  $I\subset R=k[x,y,z]$ with free resolution of the form
\begin{equation} \label{3syzygy_resolution}
	0\rar R\stackrel{\psi}\lar  R^3\stackrel{\phi}\lar  R^3 \lar R \lar R/I\rar 0,
\end{equation}
with shifts to be specified in the graded case.

Although these ideals have quite a bit of history behind us, some recent interest has been ignited through the case of the so-called $3$-{\em syzygy curves} by A. Dimca and collaborators.

Resolution (\ref{3syzygy_resolution}) seems to be quite ubiquitous under several disguises. Here are some examples:

\begin{enumerate}
	\item {\sc The self-dual case}  The resolution of a complete intersection of three elements.
	
	This is the only case where $I$ has codimension three.
	
	\item {\sc Pseudo-regular sequence} (\cite[Corollary 1.4]{Simis-triply}) Three polynomials $f_1,f_2,f_3$ in $k[x,y,z]$ generating an ideal of codimension two, such that
	$$\frac{f_1}{\gcd(f_1,f_2) \, \gcd(f_1,f_3)}, 
	\frac{f_2}{\gcd(f_1,f_2) \, \gcd(f_2,f_3)},
	\frac{f_3}{\gcd(f_1,f_3) \, \gcd(f_2,f_3)}$$
	is a regular sequence, generate an ideal with a free resolution as (\ref{3syzygy_resolution}).
	
	A notable example is the gradient ideal of an ordinary plane cusp singularity. 
	
	Actually, this result extends to the case where $R$ is a Noetherian factorial ring.
	
	\item {\sc Plane Cremona transformations} (\cite[Theorem 2.12]{HasSim2012}) Consider a simple plane Cremona map of degree $5\leq d\leq 7$ with base ideal $I\subset k[x,y,z]$. If $I$ is not saturated then its minimal free resolution is of the form (\ref{3syzygy_resolution}) (for suitable shifts).
	
	For more elaborated examples of bisimple Cremona transformations, whose base ideals have a free resolution as (\ref{3syzygy_resolution}), see \cite[Proposition 6.2.12]{GradedBook}.
	
	\item {\sc $3$-syzygy plane curve} (\cite{Dimca2024} and its references) 
	This is a general name for some classes of projective plane curves whose Jacobian (gradient) ideal has a free resolution as (\ref{3syzygy_resolution}), plus some specifics about the shifts.
	Note that, differently from the previous section, here the  equation defining the gradient  may not have a determinantal expression.
	An extended notion for hypersurfaces is discussed in \cite{Dimca_et_al2021}.
\end{enumerate}
We next discuss homological aspects of the maps in (\ref{3syzygy_resolution}), assuming without loss of generality that $I$ has height two.

Denote representative matrices of the maps in (\ref{3syzygy_resolution}) by the same respective symbols.
Up to change of coordinates and elementary operations, we may assume that the entries of $\psi$ form a regular sequence, say, $\{g_1,g_2,g_3\}$.
Then, by dualizing (\ref{3syzygy_resolution}) into $R$ gives a matrix decomposition $\phi=AK$, for some $3\times 3$ matrix $A$ over $R$, where
\begin{equation}\label{Koszul}
K:= \begin{bmatrix}
	0&       -g_3&      g_2\\
	g_3&      0 &   -g_1\\
	-g_2 & g_1 & 0
\end{bmatrix},
\end{equation}
the matrix of Koszul relations of $\{g_1,g_2,g_3\}$.

By the Buchsbaum--Eisenbud exactness criterion, $I_2(AK)=I_2(\phi)$ has height at least two.
One can pose a naive question as to what class of ideals admit such a resolution for a given regular sequence $\{g_1,g_2,g_3\}$.
This is a sort of inverse question,  tantamount to asking what are the $3\times 3$ matrices $A$ such that $I_2(AK)$ has height two. Once $AK$ is in place, in order to reach to the corresponding ideal $I$ one takes, as is well-nown, the $2\times 2$ minors of two columns of $AK$ divided by their gcd. 

We base our search on the following lemma of general nature, possibly established elsewhere before.

\begin{Lemma}\label{basic_lemma}
Let $M$ denote a $3\times 3$ matrix over  a commutative ring $S$ and let $\{g_1,g_2,g_3\}\subset S$ be any three elements with respective  matrix $G$ of Koszul relations mimicking \eqref{Koszul}. Then 
$$I_2(MG)=(g_1,g_2,g_3)(f_1,f_2,f_3),$$
where $f_1,f_2,f_3$ denote the maximal minors of the $4\times 3$ matrix $\left[\begin{matrix} 
	M \\ 
	\hline 
	g_1 \:\: g_2 \:\: g_3  \end{matrix}\right]$.
\end{Lemma}
\begin{proof}
 Letting $\Delta_{i,j}$ stand for the (non-signed) $(i,j)$-cofactor of $M$, one has: 
\begin{eqnarray*}
	{\rm adj}(MG)&=& {\rm adj}(G){\rm adj}(M)\\
	&=& \left[\begin{matrix}
		g_1^2&g_1g_2&g_1g_3\\
		g_1g_2&g_2^2&g_2g_3\\
		g1g_3&g_2g_3&g_3^2
	\end{matrix}\right]
	\left[\begin{matrix}
		\Delta_{1,1}&-\Delta_{2,1}&\Delta_{3,1}\\
		-\Delta_{1,2}&\Delta_{2,2}&-\Delta_{3,2}\\
		\Delta_{1,3}&-\Delta_{2,3}&\Delta_{3,3}
	\end{matrix}\right]
	=\left[\begin{matrix}
		g_1f_1&-g_1f_2&g_1f_3\\
		-g_2f_1&g_2f_2&-g_2f_3\\
		g_3f_1&-g_2f_2&f_3f_3
	\end{matrix}\right].
\end{eqnarray*}
Hence, $I_2(MG)=I_1({\rm adj}(MG))=(g_1,g_2,g_3)(f_1,f_2,f_3).$
\end{proof}

The following result gives a criterion for  an $A$ as required in  the previous discussion.

\begin{Proposition}\label{reverse_criterion} {\rm (Vasconcelos)}
	Let $A=(a_{i,j})$ denote a $3\times 3$ matrix over $R=k[x,y,z]$ and let $B= \left[\begin{matrix} 
		A \\ 
		\hline 
		g_1 \:\: g_2 \:\: g_3  \end{matrix}\right]$, where $g_1,g_2,g_3$ is a regular sequence in $R$.
		If $I_3(B)\subset R$ has height two, then $I_2(AK)\subset R$ has height two, where $K$ is as in \eqref{Koszul}.
\end{Proposition}
\begin{proof}  
	Apply Lemma~\ref{basic_lemma} to the case where $M=(T_{i,j})$ is the $3\times 3$ generic matrix over $k$ and $\{g_1,g_2,g_3\}=\{x,y,z\}$, with $S=k[x,y,z, T_{i,j}\,|\, 1\leq i,j\leq 3]$, to get:
	$$I_2(M \widetilde{K})=(x,y,z)(f_1,f_2,f_3),$$
	where $\widetilde{K}$ is the Koszul matrix of $\{x,y,z\}$ --
	an equality holding over $S$.
	
	Here the matrix $\widetilde{B}:=\left[\begin{matrix} 
		M \\ 
		\hline 
		x \:\: y \:\: z  \end{matrix}\right]$ is generic as well.
		Therefore, since $(f_1,f_2,f_3)$ is the subideal of maximal minors fixing the last column, its minimal primes are $(x,y,z)$ and $I_3(\widetilde{B})$ -- this follows, e.g., by noting that $(x,y,z)\det M \subset (f_1,f_2,f_3)$.
		Therefore, these are the minimal primes of $I_2(M\widetilde{K})\subset S$ as well.
		Thus 
		$$ \sqrt{(I_2(M\,\widetilde{K})}= I_3(\widetilde{B})\cap (x,y,z) S.$$
		In particular, suitable powers of $I_3(\widetilde{B})$ and $ \fm S$ intersect in $I_2(\widetilde{A}\,\widetilde{K})$.
		But this property is obviously maintained through the specialization $T_{i,j}\mapsto a_{i,j}$ and $x\mapsto g_1, y\mapsto g_2, z\mapsto g_3$.
		This gives
		$$I_3(B)^r \cap (g_1,g_2,g_3)^s \subset I_2(AK),$$
	over $R=k[x,y,z]$,	for suitable integers $r,s$.
		But since $(g_1,g_2,g_3)$ is $(x,y,z)$-primary, a sufficiently high power of $I_3(B)$ will land on $(x,y,z)$, so for suitable powers the intersection is a power of $I_3(B)$ alone.
		Then, $\Ht I_2(AK)\geq 2$.
\end{proof}

We pose:

\begin{Question}\rm
	Is the family of $3\times 3$ matrices as in Proposition~\ref{reverse_criterion} closed under certain properties, or else just a scattered bunch of members?
\end{Question}

\section{Structured sparse square matrices}\label{Sec3}

\subsection{Preliminaries}

Let $m\geq 3$  be an integer and let $\XX=(x_{i,j})_{1\leq i,j\leq m}$ be the $m\times m$ generic matrix as before. Let $0\leq r\leq s\leq m-2$ be integers.   We denote by $\mathfrak{G}_{m,r,s}=(g_{i,j})$ the $m\times m$ linear section of the $m\times m$ generic matrix $\XX$ such that 

\begin{equation}\nonumber
	g_{i,j}=\left\{\begin{array}{llc}
		x_{i,j},& \mbox{if $r+2\leq i+j\leq 2m-s$}\\
		0,&\mbox{otherwise}
	\end{array}\right.
\end{equation} 
Thus, e.g., for $m=4$, one has the following sections.
{\small
	$$\mathfrak{G}_{4,0,0}=\left[\begin{matrix}
		x_{1,1}&x_{1,2}&x_{1,3}&x_{1,4}\\
		x_{2,1}&x_{2,2}&x_{2,3}&x_{2,4}\\
		x_{3,1}&x_{3,2}&x_{3,3}&x_{3,4}\\
		x_{4,1}&x_{4,2}&x_{4,3}&x_{4,4}
	\end{matrix}\right],\quad \mathfrak{G}_{4,0,1}=\left[\begin{matrix}
		x_{1,1}&x_{1,2}&x_{1,3}&x_{1,4}\\
		x_{2,1}&x_{2,2}&x_{2,3}&x_{2,4}\\
		x_{3,1}&x_{3,2}&x_{3,3}&x_{3,4}\\
		x_{4,1}&x_{4,2}&x_{4,3}&0
	\end{matrix}\right]$$
	
	$$
	\mathfrak{G}_{4,0,2}=\left[\begin{matrix}
		x_{1,1}&x_{1,2}&x_{1,3}&x_{1,4}\\
		x_{2,1}&x_{2,2}&x_{2,3}&x_{2,4}\\
		x_{3,1}&x_{3,2}&x_{3,3}&0\\
		x_{4,1}&x_{4,2}&0&0\\
	\end{matrix}\right],\quad
	\mathfrak{G}_{4,1,1}=\left[\begin{matrix}
		0&x_{1,2}&x_{1,3}&x_{1,4}\\
		x_{2,1}&x_{2,2}&x_{2,3}&x_{2,4}\\
		x_{3,1}&x_{3,2}&x_{3,3}&x_{3,4}\\
		x_{4,1}&x_{4,2}&x_{4,3}&0
	\end{matrix}\right]$$
	
	$$
	\mathfrak{G}_{4,1,2}=\left[\begin{matrix}
		0&x_{1,2}&x_{1,3}&x_{1,4}\\
		x_{2,1}&x_{2,2}&x_{2,3}&x_{2,4}\\
		x_{3,1}&x_{3,2}&x_{3,3}&0\\
		x_{4,1}&x_{4,2}&0&0\\
	\end{matrix}\right],\quad
	\mathfrak{G}_{4,2,2}=\left[\begin{matrix}
		0&0&x_{1,3}&x_{1,4}\\
		0&x_{2,2}&x_{2,3}&x_{2,4}\\
		x_{3,1}&x_{3,2}&x_{3,3}&0\\
		x_{4,1}&x_{4,2}&0&0
	\end{matrix}\right]$$
Let $R_{m,r,s}$ denote the polynomial ring over $k$ on the nonzero entries of $\mathfrak{G}_{m,r,s}$. 	Clearly,  $\dim R_{m,r,s}=m^2-\mathfrak{o}(r)-\mathfrak{o}(s),$  where $\mathfrak{o}(u)={u+1\choose 2}$  for any integer $u\geq 1.$ 
The special case $\mathfrak{G}_{m,0,s}$ has been thoroughly discussed in \cite{Degen-Gen}.

Now, throughout the discussion $r$ and $s$ will be fixed. Therefore, in order to use a lighter notation, if no confusion arises, we set $\mathfrak{G}_m:=\mathfrak{G}_{m,r,s}$ and $R:=R_{m,r,s}$.

Note that the assumption $r\leq s$ can be turned around by easy column and row transpositions.

	\begin{Proposition}\label{mgres} 
		The ideal $I_{m-1}(\mathfrak{G}_m)\subset R$ of submaximal minors has height $4.$ In particular, $I_{m-1}(\mathfrak{G}_m)$ is a Gorenstein ideal with  minimal graded free resolution 
		{\small\begin{equation}\label{resI}
				0\to R(-2m)\to R(-m-1)^{m^2}\to R(-m)^{2m^2-2}\to R(-m+1)^{m^2}\to R
		\end{equation}}
		and the multiplicity  $e( R/I_{m-1}(\mathfrak{G}_m))$ of $ R/I_{m-1}(\mathfrak{G}_m)$ is $\frac{m^2(m+1)(m-1)}{12}.$
	\end{Proposition}
	\begin{proof}  By \cite[Proposition 2.4]{DetBook2024}, the ideal $I_{m-1}(\mathfrak{G}_m)$ is a specialization of the ideal of the generic submaximal minors. Hence, $I_{m-1}(\mathfrak{G}_m)$ is also a height $4$ Gorenstein ideal with resolution as in \eqref{resI}. The multiplicity of $R/I_{m-1}(\mathfrak{G}_m)$ is a standard calculation off the resolution.
	\end{proof}
	
\subsection{Related rational maps}

In this part we discuss some natural rational maps attached to the matrix invariants so far.

\subsubsection{The cofactor map}

Throughout, let  $\Delta(\mathfrak{G}_m)_{j,i}$  denote the  $(i,j)$-cofactor of the matrix $\mathfrak{G}_m=\mathfrak{G}_{m,r,s}$.
We wish to think of these cofactors as coordinate forms
defining a suitable rational map.
First we develop the  ground material to discuss the structure of the homogeneous coordinate ring of such a map.
As it happens, we accomplish this by taking a slight divergence from the previous linear sections as follows:

$\bullet$ First, having in mind related rational map, we switch to the $m\times m$ generic matrix 	$\YY=(y_{i,j})_{1\leq i,j\leq m}$ whose entries are the ``dual'' variables to $(x_{i,j})$; we then set $S:=k[y_{i,j}|1\leq i,j\leq m]$ throughout.

$\bullet$ Second, we now consider  integers $1\leq r\leq s\leq m$ and deal with two ladders sitting each on a suitable $(m-1)\times (m-1)$ submatrix of the $m\times m$ generic matrix $\YY$ above; we choose these submatrices to be, one obtained by deleting the first row and column, the other by deleting the last row and column. 

The ladders are pretty much like the ones $\mathfrak{G}_{m-1,r,s}$ above, with either $s=0$ or $r=0$.
Since they are ladders sitting on different matrices, we rename them anew:	
	
	$$\mathfrak{L}_{m,r}:=\left[\begin{matrix}
		&&&&&y_{2,r+1}&\cdots&y_{2,m}\\
		&&&&y_{3,r}&y_{3,r+1}&\cdots&y_{3,m}\\
		&&&&\vdots&\vdots&&\vdots\\
		&&y_{r,3}&\cdots&y_{r,r}&y_{r,r+1}&\cdots&y_{r,m}\\
		&y_{r+1,2}&y_{r+1,3}&\cdots&y_{r+1,r}&y_{r+1,r+1}&\cdots&y_{r+1,m}\\
		&\vdots&\vdots&&\vdots&\vdots&&\vdots\\
		&y_{m,2}&y_{m,3}&\cdots&y_{m,r}&y_{m,r+1}&\cdots&y_{m,m}
	\end{matrix}\right],$$

	$$\widetilde{\mathfrak{L}}_{m,s}:=\left[\begin{matrix}
		y_{1,1}&\cdots&y_{1,m-s}&y_{1,m-s+1}&\cdots&y_{1,m-2}&y_{1,m-1}\\
		\vdots&&\vdots&\vdots&&\vdots&\vdots\\
		y_{m-s,1}&\cdots&y_{m-s,m-s}&y_{m-s,m-s+1}&\cdots&y_{m-s,m-2}&y_{m-s,m-1}\\
		y_{m-s+1,1}&\cdots&y_{m-s+1,m-s}&y_{m-s+1,m-s+1}&\cdots&y_{m-s+1,m-2}\\
		\vdots&&\vdots&\vdots&&\\
		y_{m-2,1}&\cdots&y_{m-2,m-s}&y_{m-2,m-s+1}\\
		y_{m-1,1}&\cdots&y_{m-1,m-s}\end{matrix}\right].$$
	We will assume that $\widetilde{\mathfrak{L}}_{m,0}=\mathfrak{L}_{m,0}=\emptyset$ and $I_{m}(\emptyset)=0.$
	
	If no confusion arises,  denote by $I_{m-r}(\mathfrak{L}_{m,r}),$ (respectively,  $I_{m-s}(\widetilde{\mathfrak{L}}_{m,s})$)  the ideal of $S$ generated by the $(m-r)$-minors with entries inside the ladder $\mathfrak{L}_{m,r}$ (respectively, $\widetilde{\mathfrak{L}}_{m,s}$).
	
	For every $2\leq  u\leq r+1 $  and $m-s\leq v\leq m-1$ let 
	$${\bf i}_{u}:=\{r-u+3,r-u+4,\ldots,m\}, \quad {\bf j}_{u}=\{u,u+1,\ldots,m\}$$
	$$\tilde{{\bf i}}_{v}:=\{1,2,\ldots,2m-s-1-v\}, \quad \tilde{{\bf j}}_{v}=\{1,2,\ldots,v\}$$
	The submatrix  $\YY_{{\bf i}_u,{\bf j}_u}$  of $\YY$ with row (column) indices ${\bf i}_u$ (${\bf j}_u$) coincides with the  $(m+u-r-2)\times(m-u+1)$ submatrix contained in $\mathfrak{L}_{m,r}$ with ladder corners  $y_{r-u+3,u},y_{r-u+4,m},y_{m,u}, y_{m,m}.$ Quite similarly,   $\YY_{\tilde{\bf i}_v,\tilde{\bf j}_v}$  coincides with the $(2m-s-1-v)\times v$ submatrix contained in $\widetilde{\mathfrak{L}}_{m,s}$ with ladder corners   $y_{1,1},y_{1,v},y_{2m-s-1-v,1}, y_{2m-s-1-v,v}.$ 
	
	They become esential generating components of  our ladder ideals: 
	{\small
	\begin{equation}\label{ladderidealasasum}
		I_{m-r}(\mathfrak{L}_{m,r})= (I_{m-r}(\YY_{{\bf i}_u,{\bf j}_u}))_{2\leq u\leq r+1} \; \mbox{and}\;
	I_{m-s}(\widetilde{\mathfrak{L}}_{m,s})= (I_{m-s}(\YY_{\tilde{\bf i}_{v},\tilde{\bf j}_{v}}))_{m-s\leq v\leq m-1}.
	\end{equation}
}

	\begin{Example}\rm For $m=5,$ $r=3$ and $s=2$ one has
		$$\mathfrak{L}_{5,3}=\left[\begin{matrix}
			&&y_{2,4}&y_{2,5}\\
			&y_{3,3}&y_{3,4}&y_{3,5}\\
			y_{4,2}&y_{4,3}&y_{4,4}&y_{4,5}\\
			y_{5,2}&y_{5,3}&y_{5,4}&y_{5,5}
		\end{matrix}\right],\quad 
		\widetilde{\mathfrak{L}}_{5,2}=
		\left[
		\begin{array}{cccccc}
			y_{1,1}&y_{1,2}&y_{1,3}&y_{1,4}\\
			y_{2,1}&y_{2,2}&y_{2,3}&y_{2,4}\\
			y_{3,1}&y_{3,2}&y_{3,3}&y_{3,4}\\
			y_{4,1}&y_{4,2}&y_{4,3}
		\end{array}
		\right]$$

		$$\YY_{{\bf i}_2,{\bf j}_2}=
		\left[\begin{matrix}
			y_{4,2}&y_{4,3}&y_{4,4}&y_{4,5}\\
			y_{5,2}&y_{5,3}&y_{5,4}&y_{5,5}
		\end{matrix}\right],\,
		\YY_{{\bf i}_3,{\bf j}_3}=
		\left[\begin{matrix}
			y_{3,3}&y_{3,4}&y_{3,5}\\
			y_{4,3}&y_{4,4}&y_{4,5}\\
			y_{5,3}&y_{5,4}&y_{5,5}
		\end{matrix}\right],\, \YY_{{\bf i}_4,{\bf j}_4}=
		\left[\begin{matrix}
			y_{2,4}&y_{2,5}\\
			y_{3,4}&y_{3,5}\\
			y_{4,4}&y_{4,5}\\
			y_{5,4}&y_{5,5}
		\end{matrix}\right]$$
		
		$$\YY_{\tilde{\bf i}_{3},\tilde{\bf j}_{3}}= 
		\left[
		\begin{array}{cccccc}
			y_{1,1}&y_{1,2}&y_{1,3}\\
			y_{2,1}&y_{2,2}&y_{2,3}\\
			y_{3,1}&y_{3,2}&y_{3,3}\\
			y_{4,1}&y_{4,2}&y_{4,3}
		\end{array}
		\right],\, 
		\YY_{\tilde{\bf i}_{4},\tilde{\bf j}_{4}}=\left[
		\begin{array}{cccccc}
			y_{1,1}&y_{1,2}&y_{1,3}&y_{1,4}\\
			y_{2,1}&y_{2,2}&y_{2,3}&y_{2,4}\\
			y_{3,1}&y_{3,2}&y_{3,3}&y_{3,4}\\
		\end{array}
		\right],$$
		with decompositions $$I_2(\mathfrak{L}_{5,3})=(I_2(\YY_{{\bf i}_2,{\bf j}_2}), I_2(\YY_{{\bf i}_3,{\bf j}_3}), I_2(\YY_{{\bf i}_4,{\bf j}_4}))\quad \mbox{and}\quad I_2(\widetilde{\mathfrak{L}}_{5,2})=(I_2(\YY_{\tilde{\bf i}_{3},\tilde{\bf j}_{3}}), I_2(\YY_{\tilde{\bf i}_{4},\tilde{\bf j}_{4}}).$$
		
	\end{Example}

Recall that, throughout, the ground ring  was denoted $R$ and its ``dual'' was denoted $S$.
	
	\begin{Theorem}\label{ratmapminorssubmax}
		Let
		$$\mathfrak{F}_{m,r,s}: \pp^{m^2-\mathfrak{o}(r)-\mathfrak{o}(s)-1}=\Proj(R)\dasharrow \pp^{m^2-1}=\Proj(S),$$
		be the rational map whose  $(i,j)$-coordinate  is $\Delta(\mathfrak{G}_m)_{j,i}.$
		One has$:$
		\begin{enumerate}
			\item[\rm (a)] $\mathfrak{F}_{m,r,s}$ is a birational map onto its image.
			\item[\rm (b)] Both $I_{m-r}(\mathfrak{L}_{m,r})$ and $I_{m-s}(\widetilde{\mathfrak{L}}_{m,s})$ are contained in the homogeneous defining ideal of the image of $\mathfrak{F}_{m,r,s}.$
		\end{enumerate}	
	\end{Theorem}
	\begin{proof} (a) As a fallout of  \cite[Proposition 2.3]{DetBook2024}, the Krull dimension of the $k$-algebra
	$k[\Delta(\mathfrak{G}_m)_{j,i}|1\leq i,j\leq m]$ is $m^2-\mathfrak{o}(r)-\mathfrak{o}(s)-1$, and by Proposition \ref{mgres}, we know that the base ideal $I_{m-1}(\mathfrak{G}_m)$ of $\mathfrak{F}_{m,r,s}$ has linear rank $m^2-1.$ 
	Thus, the assertion follows from \cite[Theorem 3.2]{AHA}.
	
		(b) Write $\mathfrak{G}:=\mathfrak{G}_{m,r,s}$ for lightness. With the notation of Setup~\ref{Setup_cofactor}, for every $2\leq u\leq r+1$, the $(m-u+1)\times(r-u+2)$ submatrix $\mathfrak{G}_{{\bf j}_u,\mathfrak{c}({\bf i}_u)}$ of $\mathfrak{G}$ contains an anti-diagonal formed by the variables $x_{u,r-u+2},x_{u+1,r-u+1},\ldots,x_{r+1,1}.$ Thus, since the entries of $\mathfrak{G}_{{\bf j}_u,\mathfrak{c}({\bf i}_u)}$ are mutually distinct variables or zeros, it follows that $\rk \mathfrak{G}_{{\bf j}_u,\mathfrak{c}({\bf i}_u)}=r-u+2.$ Also, $\mathfrak{G}_{\mathfrak{c}({\bf j}_u),\mathfrak{c}({\bf i}_u)}=\boldsymbol0_{(u-1)\times(r-u+2) }.$ Hence, by   Proposition~\ref{Image_map_minors}, the ideal	$$I_{m-u+1-(r-u+2)+1}(\YY_{{\bf i}_u,{\bf j}_u})=I_{m-r}(\YY_{{\bf i}_u,{\bf j}_u})$$
	is contained in  the defining ideal of the image of $\mathfrak{F}_{m,r,s}.$ Therefore,  \eqref{ladderidealasasum} implies that $I_{m-r}(\mathfrak{L}_{m,r})$ is also contained in  the defining ideal of the image of $\mathfrak{F}_{m,r,s}$, as stated. 
	
	The argument for  $I_{m-s}(\widetilde{\mathfrak{L}}_{m,s})$ is similar.
\end{proof}

	\subsubsection{The polar map}
	
For the few geometric notions used throughout, such as the Hessian matrix and the Hessian determinant, as well as the polar map, we refer to \cite[Chapter 3]{DetBook2024}.
The notion of a homaloidal form has been mentioned before in Subsection~\ref{lin_syz_II}.
	
	\begin{Theorem}\label{Hessian_and_homaloid}
		Let $f_{m,r,s}\in R$ be the determinant of the matrix  $\mathfrak{G}_m=\mathfrak{G}_{m,r,s}.$ Then$:$
		\begin{enumerate}
			\item[\rm(a)] If $r\neq s$ then $f_{m,r,s}$ has vanishing Hessian.
			\item[\rm(b)] If $r=s$ then $f_{m,r,r}$ is homaloidal.
		\end{enumerate}
	\end{Theorem}
	\begin{proof} (a) In this case, $r<s.$ Consider for example the following $(m-s)\times(m-s)$ submatrix of the ladder $\widetilde{\mathfrak{L}}_{m,s}$
	
	$$M=\left[\begin{matrix}
		y_{s,1}&\cdots&y_{s,s}\\
		\vdots&&\vdots\\
		y_{m-1,1}&\cdots&y_{m-1,s}
	\end{matrix}\right]$$
	Thus, 
	$$
	\det\left[\begin{matrix}
		\frac{\partial f_{m,r,s} }{\partial x_{s,1}}&\cdots&\frac{\partial f_{m,r,s} }{\partial x_{s,s}}\\
		\vdots&&\vdots\\
		\frac{\partial f_{m,r,s} }{\partial x_{1,m-1}}&\cdots&\frac{\partial f_{m,r,s} }{\partial x_{s,m-1}}
	\end{matrix}\right]=
	\det\left[\begin{matrix}
		\Delta(\mathfrak{G}_m)_{s,1}&\cdots&\Delta(\mathfrak{G}_m)_{s,s}\\
		\vdots&&\vdots\\
		\Delta(\mathfrak{G}_m)_{1,m-1}&\cdots&\Delta(\mathfrak{G}_m)_{m-1,s}
	\end{matrix}\right]=0
	$$
	where the first equality follows from Proposition \ref{GolMar} and the second equality follows from  Theorem~\ref{ratmapminorssubmax} and the fact that 
	$$M=\left[\begin{matrix}
		y_{1,s}&\cdots&y_{1,m-1}\\
		\vdots&&\vdots\\
		y_{m-s,s}&\cdots&y_{m-s,m-1}
	\end{matrix}\right]$$
	is a $(m-s)\times (m-s)$ submatrix of  $\widetilde{\mathfrak{L}}_{m,s}.$

	(b) To simplify the notation, we will write $f=f_{m,r,r}$ and $\Delta_{i,j}=\Delta(\mathfrak{G}_{m,r,r})_{i,j}.$ 
	
	We prove the field equality
	$$k(\partial f/\partial x_{i,j}\,|\; r+2\leq i+j\leq 2m-r)=k(\Delta_{i,j}\, |\,1\leq i,j\leq m).$$
	Then the result follows from Theorem~\ref{ratmapminorssubmax}(a).
	
	To this effect, for every integer $0\leq u\leq r $ define 
	$$\mathcal{I}_{u}:=\{(i,j)\,|\,r+2-u\leq i+j\leq 2m-r+u\}\quad\mbox{and}\quad K_u:=k(\Delta_{i,j};\,(i,j)\in \mathcal{I}_u).$$
	
	Since $K_0=k(\partial f/\partial x_{i,j}\,|\; r+2\leq i+j\leq 2m-r)$ and $K_r=k(\Delta_{i,j}\, |\,1\leq i,j\leq m)$, it suffices to show that $K_u=K_{u+1}$ for every $0\leq u\leq r-1.$
	
	Induct on $u$.
	By definition, we have $K_{u}\subset K_{u+1}.$ Moreover, also by definition, to show that the inclusion $K_{u}\subset K_{u+1}$ is an equality it is enough to show that $\Delta_{i,j}\in K_u$ for every $(i,j)\in \mathcal{I}_{u+1}\setminus \mathcal{I}_u.$ But, by symmetry, its is enough to show that $\Delta_{i,j}\in K_u$ for every $(i,j)\in \mathcal{I}_{u+1}$ such that $i+j=r+2-u-1.$
	
	Let $1\leq i,j\leq m$ be integers such that $i+j=r+2-u-1.$ Set $H\in I_{m-r}(\widetilde{\mathfrak{L}}_m)S$ for  the determinant of the  following $(m-r)\times (m-r)$ submatrix of $\mathfrak{L}_m:$
	
	$$\left[\begin{matrix}
		y_{i,j}&y_{i,r+2-i}&\cdots&y_{i,r+2-i+s}&\cdots&y_{i,r+2-i+\delta}\\
		y_{r+2-j,j}&y_{r+2-j,r+2-i}&\cdots&y_{r+2-j,r+2-i+s}&\cdots &y_{r+2-j,r+2-i+\delta}\\
		\,\,\,\,\vdots&\,\,\,\,\vdots&\cdots&\,\,\,\,\vdots&\,\,\,\,\cdots&\,\,\,\,\vdots\\
		y_{r+2-j+t,j}&y_{r+2-j+t,r+2-i}&\cdots&y_{r+2-j+t,r+2-i+s}&\cdots &y_{r+2-j+t,r+2-i+\delta}\\
		\,\,\,\,\vdots&\,\,\,\,\vdots&\cdots&\,\,\,\,\vdots&\cdots&\,\,\,\,\vdots\\
		y_{r+2-j+\delta,j}&y_{r+2-j+\delta,r+2-i}&\cdots&y_{r+2-j+\delta,r+2-i+s}&\cdots&y_{r+2-j+\delta,r+2-i+\delta}
	\end{matrix}\right]$$
	where $\delta:=m-r-2.$  With the exception of  $y_{i,j}$, all other entries of this matrix belong to the set of variables $\{y_{i,j}\,|(i,j)\in\mathcal{I}_{u}\}$. Thus, by  Laplace, we can write $H$ as 
	$$H=y_{i,j}H_1+H_2$$
	where $H_1$ and $H_2$ are homogeneous polynomials of degrees $m-r-1$ and $m-r,$ respectively, that depend only on the variables belonging to the set $\{y_{i,j}\,|(i,j)\in\mathcal{I}_{u}\}.$ 
	Thus, $H_1(\boldsymbol\Delta), H_2(\boldsymbol\Delta)\in K_{u}$, where $\boldsymbol\Delta=\{\Delta_{j,i}\,|(i,j)\in\mathcal{I}_{u}\}.$ In particular,  since
	$$\Delta_{i,j}H_{1}(\boldsymbol\Delta)+H_2(\boldsymbol\Delta)=0 \quad (\mbox{because $H\in I_{m-r}(\widetilde{\mathfrak{L}}_m)S$ and by Theorem~\ref{ratmapminorssubmax}(b)})$$
	and $H_1(\boldsymbol\Delta)\neq 0$ by Corollary~\ref{Defining_ideal_image} below -- whose proof runs independently --  we have  $\Delta_{i,j}\in K_u.$ 
	This establishes the sought equality $K_u=K_{u+1}.$ 
	\end{proof}
	
		\subsection{The balanced case: $r=s$}
		
			Throughout assume that $r=s$.
		In the previous subsection we have stated that $f_{m,r,r}$ is homaloidal, but the argument depended in knowing the initial degree of the homogeneous defining ideal of the image. This brings us to the next part.
		
		\subsubsection{The structure of the image}

	We will draw upon the following result,  a straightforward adaptation of \cite[Theorem 1]{Sturmfels1990}.
	
	\begin{Proposition}\label{anti-diagonals}
		Let $\mathcal{G}=(x_{i,j})$ denote an $m\times n$ generic matrix over a field $k$ and let  $1\leq t\leq \min{m,n}$.
		Let $\prec$ stand for the reverse lexicographic term order on $k[x_{i,j}]$ induced by the following order of the variables:
		
		$\displaystyle x_{1,1}> x_{1,2}>\cdots > x_{1,n}>x_{2,1}>\cdots > x_{2,n}>\cdots >x_{m,1}> \cdots > x_{m,n}.$
		
		Then:
		
		{\rm (i)} The initial term of any  $t\times t$-minor with respect to $\prec$ is the product of the entries along its main anti-diagonal.
		
		{\rm (ii)} The set of $t\times t$-minors of $\mathcal{G}$ is a Gr\"obner basis of the ideal $I_t(\mathcal{G})$ with respect to $\prec$.
	\end{Proposition}
	\begin{proof}
		We borrow from \cite[Theorem 1]{Sturmfels1990}, the notation $[a_1,\ldots,a_t|b_1,\ldots,b_t]$ for the $t\times t$-minor with rows and columns indexed by $a_1<\cdots <a_t$ and $b_1<\cdots <b_t$, respectively.
		
		Letting $f:=k[x_{i,j}]\rar k[x_{i,j}]$ denote the $k$-algebra automorphism such that $f(x_{i,j})= x_{i,m+1-j}$, for every pair of indices $i,j$, it obtains
		\begin{eqnarray}
			f\,({\rm in}_{\prec'}([a_1,\ldots,a_t|b_1,\ldots,b_t]))&=&f\,(x_{a_1,b_1}\cdots x_{a_t,b_t})\nonumber\\
			&=&x_{a_1,m+1-b_1}\cdots x_{a_t,m+1-b_t}\nonumber\\
			&=&{\rm in}_{\prec}([a_1,\ldots,a_t|m+1-b_t,\ldots,m+1-b_1])\nonumber
		\end{eqnarray}
		Now, for any subset $\mathfrak{P}\subset k[x_{i,j}]$ write $f(\mathfrak{P})=\{f(p)| p\in \mathfrak{P}\}$.
		Writing $\mathfrak{D}_t$ for the set of all $t\times t$-minors of $\mathcal{G}$, set ${\rm in}_{\_}(\mathfrak{D}_t)$ for the monomial ideal generated by the initial terms of the elements of $\mathfrak{D}_t$ in a given term order.
		Then the above display tells us that
		$f({\rm in}_{\prec'}(\mathfrak{D}_t))={\rm in}_{\prec}(\mathfrak{D}_t).$ 
		On the other hand, the result in \cite{Sturmfels1990} says that the initial ideal ${\rm in}_{\prec'}(I_t(\mathcal{G)})$ coincides with ${\rm in}_{\prec'}(\mathfrak{D}_t)$.
		Moreover, the Hilbert functions of the four ideals 
		$$I_t(\mathcal{G}),\;{\rm in}_{\prec'}(I_{t}(\mathcal{G})),\;{\rm in}_{\prec}(I_{t}(\mathcal{G})),\;f\,({\rm in}_{\prec'}(I_{t}(\mathcal{G}))$$
		are the same.
		Therefore, the inclusion ${\rm in}_{\prec}(\mathfrak{D}_t)\subset {\rm in}_{\prec}(I_{t}(\mathcal{G}))$ must be an equality, as was to be shown.
	\end{proof}
	
	Since we are assuming that $r=s$, write $\mathfrak{L}_{m}:=\mathfrak{L}_{m,r}$ and   $\widetilde{\mathfrak{L}}_{m}:=\widetilde{\mathfrak{L}}_{m,r}$ for lighter reading.
	\begin{Proposition}\label{height} The ideal 
		$(I_{m-r}(\mathfrak{L}_m),I_{m-r}(\widetilde{\mathfrak{L}}_m))$ of $S$ has height $2{r+1\choose 2}.$
	\end{Proposition}
	\begin{proof} By \cite[Theorem 4.6 and Corollary 4.7]{HT}, $\Ht I_{m-r}(\mathfrak{L}_m) = \Ht I_{m-r}(\widetilde{\mathfrak{L}}_m)={r+1\choose 2}.$ 
		
		Let $\prec$ be the monomial order as in Proposition~\ref{anti-diagonals}.  By the latter and \cite[Lemma 4.5.1]{GMiller}, the set of $(m-r)$-minors of $\mathfrak{L}_m$ (resp., $\widetilde{\mathfrak{L}}_m$) is a Gr\"obner basis of the ideal $I_{m-r}(\mathfrak{L}_m)$  (resp., $I_{m-r}(\widetilde{\mathfrak{L}}_m)$) with respect to $\prec$.  
		
		Note that: 
		
		(A) The diagonal of the $(m-r)$-minors in $\mathfrak{L}_{m}$ involves  only  variables $y_{i,j}$ such that $i+j>m+2;$ 
		
		(B) The diagonal of the $(m-r)$-minors in $\widetilde{\mathfrak{L}}_{m}$ involves only  variables $y_{i,j}$ such that $i+j<m+2.$	
		
		From (A), (B)  and Proposition~\ref{anti-diagonals} we have
		
		$$\Ht ({\rm in}(I_{m-r}(\mathfrak{L}_m)),{\rm in}(I_{m-r}(\widetilde{\mathfrak{L}}_m)))=\Ht{\rm in}(I_{m-r}(\mathfrak{L}_m))+\Ht {\rm in}(I_{m-r}(\widetilde{\mathfrak{L}}_m)).$$
		
		Hence,
		$$\Ht ({\rm in}(I_{m-r}(\mathfrak{L}_m)),{\rm in}(I_{m-r}(\widetilde{\mathfrak{L}}_m)))=2{r+1\choose 2}.$$
		
		Since ${\rm in}(I_{m-r}(\mathfrak{L}_m),I_{m-r}(\widetilde{\mathfrak{L}}_m))\supset( {\rm in}(I_{m-r}(\mathfrak{L}_m)),{\rm in}(I_{m-r}(\widetilde{\mathfrak{L}}_m)))$ we have 
		
		$$\Ht{\rm in}(I_{m-r}(\mathfrak{L}_m),I_{m-r}(\widetilde{\mathfrak{L}}_m))\geq 2{r+1\choose 2}.$$
		Thus,
		$$2{r+1\choose 2}\geq \Ht (I_{m-r}(\mathfrak{L}_m),I_{m-r}(\widetilde{\mathfrak{L}}_m))=\Ht{\rm in}(I_{m-r}(\mathfrak{L}_m),I_{m-r}(\widetilde{\mathfrak{L}}_m))\geq 2{r+1\choose 2}.$$
		Consequently,  $\Ht (I_{m-r}(\mathfrak{L}_m),I_{m-r}(\widetilde{\mathfrak{L}}_m))=2{r+1\choose 2}.$
	\end{proof}

	\begin{Theorem}\label{r=s_is_Gorenstein}
		The ring  $S/(I_{m-r}(\mathfrak{L}_m),I_{m-r}(\widetilde{\mathfrak{L}}_m))$ is Gorenstein.
	\end{Theorem}
	\begin{proof} First, note that $\mathcal{Y}=\{y_{i,j}\,|\,2\leq i,j\leq m-1,\, r+3\leq i+j\leq 2m-r-1\}$ is  the set of entries common to the ladders $\mathfrak{L}_{m}$  and $\widetilde{\mathfrak{L}}_{m}$.
	
	This suggests a deformation procedure in order to separate variables.
	Thus, let $\tt$ be a set of new variables $\{t_{i,j}\,|\,2\leq i,j\leq m-1,\, r+3\leq i+j\leq 2m-r-1\}$ and define the following polynomial subrings of $S[\tt]$ 
	
	$\bullet$ $S_1$ is the polynomial ring over $k$ in the entries of  $\mathfrak{L}_m$ and the variables $y_{1,m},y_{m,1}.$
	
	$\bullet$  $S_2$ is the polynomial ring over $k$ in the entries of   $\widetilde{\mathfrak{L}}_{m}.$ 
	
	$\bullet$ $S_3$ is the polynomial ring over $k$ in the variables $\tt\cup (\widetilde{\mathfrak{L}}_{m}\setminus \mathcal{Y})$


	We introduce  the following ``deformation" $\widetilde{\mathfrak{L}}_{m}(\tt)$ of $\widetilde{\mathfrak{L}}_{m}$:

	$$\scriptsize{\left[\begin{matrix}
			\eta(y_{1,1})&\cdots&\eta(y_{1,m-r})&\eta(y_{1,m-r+1})&\cdots&\eta(y_{1,m-2})&\eta(y_{1,m-1})\\
			\,\,\,\,\,\,\vdots&&\,\,\,\,\,\,\vdots&\,\,\,\,\,\,\vdots&&\,\,\,\,\,\,\vdots&\,\,\,\,\,\,\vdots\\
			\eta(y_{m-r,1})&\cdots&\eta(y_{m-r,m-r})&\eta(y_{m-r,m-r+1})&\cdots&\eta(y_{m-r,m-2})&\eta(y_{m-r,m-1})\\
			\eta(y_{m-r+1,1})&\cdots&\eta(y_{m-r+1,m-r})&\eta(y_{m-r+1,m-r+1})&\cdots&\eta(y_{m-r+1,m-2})\\
			\,\,\,\,\,\,\vdots&&\,\,\,\,\,\,\vdots&\,\,\,\,\,\,\vdots&&\\
			\eta(y_{m-2,1})&\cdots&\eta(y_{m-2,m-r})&\eta(y_{m-2,m-r+1})\\
			\eta(y_{m-1,1})&\cdots&\eta(y_{m-1,m-r})\end{matrix}\right]}$$
	where $\eta$ is the isomorphism of $k$-algebras
	$$\eta:S_2\to S_3,\quad \mbox{ $y_{i,j}\mapsto t_{i,j}$ if $y_{i,j}\in\mathcal{Y}$ and $y_{i,j}\mapsto y_{i,j}$ if $y_{i,j}\in \widetilde{\mathfrak{L}}_{m}\setminus \mathcal{Y}.$}$$
	
	Since	
	$$\frac{S[\tt]}{(I_{m-r}(\mathfrak{L}_m),I_{m-r}(\widetilde{\mathfrak{L}}_{m}(\tt)))}\simeq\frac{S_1}{I_{m-r}(\mathfrak{L}_m)}\otimes_k \frac{S_3}{I_{m-r}(\widetilde{\mathfrak{L}}_{m}(\tt))}$$
	and $S_1/I_{m-r}(\mathfrak{L}_m)$ and $S_3/I_{m-r}(\widetilde{\mathfrak{L}}_{m}(\tt))$ are both Gorenstein of codimension ${r+1\choose 2}$, then 
	$$\frac{S[\tt]}{(I_{m-r}(\mathfrak{L}_m),I_{m-r}(\widetilde{\mathfrak{L}}_{m}(\tt)))}$$
	is also Gorenstein of dimension $m^2+|\tt|-2{m-1\choose 2}$ (where $|\tt|$ is the cardinality of the set $\tt$).
	
	Now, let  $\tt-\mathcal{Y}$ be the sequence $$\{t_{i,j}-y_{i,j}\,|\,2\leq i,j\leq m-1,\, r+3\leq i+j\leq 2m-r-1\}.$$
	
	Note that  $$\frac{S}{(I_{m-r}(\mathfrak{L}_m),I_{m-r}(\widetilde{\mathfrak{L}}_m))}\simeq \frac{S[\tt]}{(\tt-\mathcal{Y},I_{m-r}(\mathfrak{L}),I_{m-r}(\widetilde{\mathfrak{L}}_{m}(\tt)))}.$$
	The latter is Gorenstein because $S[\tt]/(I_{m-r}(\mathfrak{L}_m),I_{m-r}(\widetilde{\mathfrak{L}}_{m}(\tt)))$ is Gorenstein and, by Proposition~\ref{height},
	$$\dim \frac{S[\tt]}{(\tt-\mathcal{Y},I_{m-r}(\mathfrak{L}),I_{m-r}(\widetilde{\mathfrak{L}}_{m}(\tt)))}=\dim\frac{S[\tt]}{(I_{m-r}(\mathfrak{L}),I_{m-r}(\widetilde{\mathfrak{L}}_{m}(\tt)))}-|t|.$$
	Thus, 
	$S/(I_{m-r}(\mathfrak{L}_m),I_{m-r}(\widetilde{\mathfrak{L}}_m))$
	is Gorenstein too.
	\end{proof}
	
	\smallskip
	For the next theorem we make use of the following lemmata:
	
	\begin{Lemma}{\rm (Extension principle)}\label{lemmaextension} Let $S\subset T$ denote a ring extension of  $S$, and let $U$ stand for an invertible  $(m-r)\times(m-r)$ matrix over $T.$
		Then,
		
		$$I_{m-r}(\mathfrak{L}_{m,r})T=I_{m-r}(U\mathfrak{L}_{m,r})T\quad \mbox{and}\quad I_{m-r}(\widetilde{\mathfrak{L}}_{m,r})T=I_{m-r}(\widetilde{\mathfrak{L}}_{m,r}U)T,$$ 
		as ideals in $T$, where 
		
		{\small 
		$$U\mathfrak{L}_{m,r}=\left[\begin{matrix}
			&&&&y_{2,r+1}&\cdots&y_{2,m}\\
			&&&y_{3,r}&y_{3,r+1}&\cdots&y_{3,m}\\
			&&&\vdots&\vdots&&\vdots\\
			&y_{r,3}&\cdots&y_{r,r}&y_{r,r+1}&\cdots&y_{r,m}\\
			\hline
			&&&&&&\\
			&&&&U\, \YY_{{\bf i}_{2},{\bf j}_2}&&\\
			&&&&&&
		\end{matrix}\right]$$}
		
		and
		
		{\small
		$$\widetilde{\mathfrak{L}}_{m,r}U=\left[\begin{array}{ll|lllllllllllll}
			&&&y_{1,m-r+1}&\cdots&y_{1,m-2}&y_{1,m-1}\\
			&&&\vdots&&\vdots&\vdots\\
			&&&y_{m-r,m-r+1}&\cdots&y_{m-r,m-2}&y_{m-r,m-1}\\
			&\YY_{\tilde{\bf i}_{m-r},\tilde{\bf j}_{m-s}}\, U&&y_{m-r+1,m-r+1}&\cdots&y_{m-r+1,m-2}\\
			&&&\vdots&&\\
			&&&y_{m-2,m-r+1}\\
			&&\end{array}\right],$$	}
		as ladders over $T$.
	\end{Lemma}
	\begin{proof}
		For any indices $2\leq u\leq r+1$ and $m-s\leq v\leq m-1$, one has
		$$I_{m-r}(Y_{{\bf i}_{u},{\bf j}_u})T=I_{m-r}\left(\left[\begin{matrix}\mathbb{I}_{u-2}&\\&U\end{matrix}\right]Y_{{\bf i}_{u},{\bf j}_u}\right)T$$
		and 
		$$I_{m-r}(\YY_{\tilde{\bf i}_{v},\tilde{\bf j}_{v}})T=I_{m-r}\left(\YY_{\tilde{\bf i}_{v},\tilde{\bf j}_{v}}\left[\begin{matrix}U&\\&\mathbb{I}_{v-m+r}\end{matrix}\right]\right)T.$$
		
		Then the decompositions in  \eqref{ladderidealasasum} imply
		the equalities
		$$I_{m-r}(\mathfrak{L}_{m,r})T=I_{m-r}(U\mathfrak{L}_{m,r})T\quad \mbox{and}\quad I_{m-r}(\widetilde{\mathfrak{L}}_{m,r})T=I_{m-r}(\widetilde{\mathfrak{L}}_{m,r}U)T$$
		as claimed.
	\end{proof}

	\begin{Lemma}\label{regular_over_quotieni}
		Let $A$ be the  common $(m-r-1)\times (m-r-1)$ submatrix   of $\mathfrak{L}_{m,r}$ and $\widetilde{\mathfrak{L}}_{m,r}$ that has  $\{y_{i,j}\,|\, 2\leq j\leq m-r, i+j=m+1\}$ as the entries of its anti-diagonal. 
		Then $\delta:=\det A$  is regular over $S/(I_{m-r}(\mathfrak{L}_{m,r}),I_{m-r}(\widetilde{\mathfrak{L}}_{m.r}))$.
	\end{Lemma}
	\begin{proof}
		By Theorem~\ref{r=s_is_Gorenstein}, $S/(I_{m-r}(\mathfrak{L}_m),I_{m-r}(\widetilde{\mathfrak{L}}_m))$ is in particular  Cohen-Macaulay.
		Therefore, it suffices to show that  
		$$\Ht(\delta,I_{m-r}(\mathfrak{L}_{m,r}),I_{m-r}(\widetilde{\mathfrak{L}}_{m,r}))=1+\Ht(I_{m-r}(\mathfrak{L}_{m,r}),I_{m-r}(\widetilde{\mathfrak{L}}_{m,r}).$$
		
		As in the proof of the Proposition \ref{height}, let $\prec$ be the monomial order in Proposition~\ref{anti-diagonals}. Once more, drawing upon the fact that the set of $m-r$-minors of $\mathfrak{L}_{m,r}$ (resp., $\widetilde{\mathfrak{L}}_{m,r}$) is a Gr\"obner basis of the ideal $I_{m-r}(\mathfrak{L}_{m,r})$  (resp., $I_{m-r}(\widetilde{\mathfrak{L}}_{m,r})$) with respect to $\prec$, and that no anti-diagonal of any $m-r$-minor of $\mathfrak{L}_{m,r}$ (resp., of $\widetilde{\mathfrak{L}}_{m,r}$)  involves the variables $y_{i,j}$ such that $2\leq j\leq m-r$ and $i+j=m+1$, then Proposition~\ref{anti-diagonals} yields
		\begin{eqnarray}
			\Ht(\delta,I_{m-r}(\mathfrak{L}_{m,r}),I_{m-r}(\widetilde{\mathfrak{L}}_{m,r}))&\geq& 1+ \Ht {\rm in}(I_{m-r}(\mathfrak{L}_{m,r}))+ \Ht (I_{m-r}(\widetilde{\mathfrak{L}}_{m,r}))\nonumber\\
			&=&1+\Ht(I_{m-r}(\mathfrak{L}_{m,r}),I_{m-r}(\widetilde{\mathfrak{L}}_{m,r}).\nonumber
		\end{eqnarray}
		The reverse inequality
		always holds.
	\end{proof}
	
	\begin{Theorem}\label{is_a_domain}
		The ring $S/(I_{m-r}(\mathfrak{L}_{m,r}),I_{m-r}(\widetilde{\mathfrak{L}}_{m,r}))$ is a domain.
	\end{Theorem}
	\begin{proof}
		We induct on $m$.
		
		Decompose the matrices $\YY_{{\bf i}_2,{\bf j}_2}$ and $\YY_{\tilde{\bf i}_{m-r},\tilde{\bf j}_{m-r}}$ as:
		
		$$\YY_{{\bf i}_2,{\bf j}_2}=
		\left[\begin{matrix}
			A&B\\
			C&D
		\end{matrix}\right]
		\quad\mbox{and}\quad
		\YY_{\tilde{\bf i}_{m-r},\tilde{\bf j}_{m-r}}=
		\left[\begin{matrix}
			\widetilde{C}&\widetilde{B}\\
			\widetilde{D}&A
		\end{matrix}\right],$$
		where $A$ is as in Lemma~\ref{regular_over_quotieni}.
		Set again $\delta:=\det A$.
		
		Now, introduce the following invertible $(m-r-1)\times(m-r-1)$ matrices $U$ and $V$  over the ring extension $T:=S[\delta^{-1}]\subset k(\YY)$, contained in the purely trascendental quotient field of $S=k[\YY]$:
		
		$$U=\left[\begin{matrix}
			I_{m-r-1}&\boldsymbol{0}\\
			-CA^{-1}& 1
		\end{matrix}\right]
		\quad\mbox{and}\quad
		V=\left[\begin{matrix}
			1&\boldsymbol{0}\\
			-A^{-1}\widetilde{D}& I_{m-r-1}
		\end{matrix}\right]  $$
		Thus
		
		$$U\, \YY_{{\bf i}_2,{\bf j}_2}= \left[\begin{matrix}
			A&B\\
			0&D-CBA^{-1}
		\end{matrix}\right]
		\quad\mbox{and}\quad
		\YY_{\tilde{\bf i}_{m-r},\tilde{\bf j}_{m-r}}\, V=
		\left[\begin{matrix}
			\widetilde{C}-\widetilde{B}A^{-1}\widetilde{D}&\widetilde{B}\\
			\boldsymbol{0}&A
		\end{matrix}\right]$$
		
		Note that $I_{m-r}(\YY_{{\bf i}_2,{\bf j}_2})=I_1(D-CBA^{-1})$ and $I_{m-r}(\YY_{\tilde{\bf i}_{m-r},\tilde{\bf j}_{m-r}})=I_1(\widetilde{C}-\widetilde{B}A^{-1}\widetilde{D}).$
		
		From this and Lemma~\ref{lemmaextension} one has
		$$I_{m-r}(\mathfrak{L}_{m,r})S[\delta^{-1}]=(I_1(D-CBA^{-1}),I_{m-r}(\mathfrak{L}_{m-1,r-1}))S[\delta^{-1}]$$
		and
		$$I_{m-r}(\widetilde{\mathfrak{L}}_{m,r})S[\delta^{-1}]=(I_1(\widetilde{C}-\widetilde{B}A^{-1}\widetilde{D}),I_{m-r}(\widetilde{\mathfrak{L}}_{m-1,r-1}))S[\delta^{-1}]$$
		Hence,
		
		$$\frac{S[\delta^{-1}]}{
			(I_{m-r}(\mathfrak{L}_{m,r}),I_{m-r}(\widetilde{\mathfrak{L}}_{m,r}))S[\delta^{-1}]}\simeq \frac{S'[\delta^{-1}]}{(I_{m-r}(\mathfrak{L}_{m-1,r-1}),I_{m-r}(\widetilde{\mathfrak{L}}_{m-1,r-1}))S'[\delta^{-1}]}$$
		where $S'\subset S$ is the sub-polynomial ring over $k$ in the variables of the set 
		$$\{y_{i,j}\,|\,1\leq i,j\leq m\}\setminus \{y_{i,j}\,|\,y_{i,j}\in D\,\,\mbox{or}\,\, y_{i,j}\in\widetilde{C}\}$$ and 
		
		$$\mathfrak{L}_{m-1,r-1}:=\left[\begin{matrix}
			&&&&&y_{2,r+1}&\cdots&y_{2,m}\\
			&&&&y_{3,r}&y_{3,r+1}&\cdots&y_{3,m}\\
			&&&&\vdots&\vdots&&\vdots\\
			&&y_{r,4}&\cdots&y_{r,r}&y_{r,r+1}&\cdots&y_{r,m}\\
			&y_{r+1,3}&y_{r+1,4}&\cdots&y_{r+1,r}&y_{r+1,r+1}&\cdots&y_{r+1,m}\\
			&\vdots&\vdots&&\vdots&\vdots&&\vdots\\
			&y_{m-1,3}&y_{m-1,4}&\cdots&y_{m-1,r}&y_{m-1,r+1}&\cdots&y_{m-1,m}
		\end{matrix}\right],$$

		$$\widetilde{\mathfrak{L}}_{m-1,r-1}:=
		\left[\begin{matrix}
			y_{1,2}&\cdots&y_{1,m-r}&y_{1,m-r+1}&\cdots&y_{1,m-2}&y_{1,m-1}\\
			\vdots&&\vdots&\vdots&&\vdots&\vdots\\
			y_{m-r,2}&\cdots&y_{m-r,m-r}&y_{m-r,m-r+1}&\cdots&y_{m-r,m-2}&y_{m-r,m-1}\\
			y_{m-r+1,2}&\cdots&y_{m-r+1,m-r}&y_{m-r+1,m-r+1}&\cdots&y_{m-r+1,m-2}\\
			\vdots&&\vdots&\vdots&&\\
			y_{m-3,2}&\cdots&y_{m-3,m-r}&y_{m-2,m-r+1}\\
			y_{m-2,2}&\cdots&y_{m-2,m-r}\end{matrix}\right].$$
		
		By this isomorphism and the inductive step, the ring
		$$\frac{S[\delta^{-1}]}{
			(I_{m-r}(\mathfrak{L}_{m,r}),I_{m-r}(\widetilde{\mathfrak{L}}_{m,r}))S[\delta^{-1}]}$$
		is a domain. 
		
		Now apply Lemma~\ref{regular_over_quotieni} to conclude that 
		$S/(I_{m-r}(\mathfrak{L}_{m,r}),I_{m-r}(\widetilde{\mathfrak{L}}_{m.r}))$ is a domain.
	\end{proof}
	
	\begin{Corollary}\label{Defining_ideal_image}
		The defining ideal of the image of the birational map $\mathfrak{F}_{m,r,r}$ is generated by $I_{m-r}(\mathfrak{L}_{m,r})$ and $I_{m-r}(\widetilde{\mathfrak{L}}_{m,r}).$
	\end{Corollary}
	\begin{proof}
		The image of $\mathfrak{F}_{m,r,r}$ has dimension $m^2-2{r+1 \choose 2}-1$, hence its defining ideal has codimension $2{r+1 \choose 2}$.
		The result now follows from Proposition~\ref{height} and Theorem~\ref{is_a_domain}.
	\end{proof}

\subsubsection{The dual variety}

In this part we deal with the dual variety of the determinant of the matrix $\mathfrak{G}_{m,r,s}$ in the balanced case assuming moreover that $r=s=m-2$.

Recall that
$$\mathfrak{G}_{m,m-2,m-2}=
\left[\begin{matrix}
	0&0&0&\cdots&0&0&x_{1,m-1}&x_{1,m}\\
	0&0&0&\cdots&0&x_{2,m-2}&x_{2,m-1}&x_{2,m}\\
	0&0&0&\cdots&x_{3,m-3}&x_{3,m-2}&x_{3,m-1}&0\\
	\vdots&\vdots&\vdots&\iddots&\vdots&\vdots&\vdots&\vdots\\
	0&0&x_{m-3,3}&\cdots&0&0&0&0\\
	0&x_{m-2,2}&x_{m-2,3}&\cdots&0&0&0&0\\
	x_{m-1,1}&x_{m-1,2}&x_{m-1,3}&\cdots&0&0&0&0\\
	x_{m,1}&x_{m,2}&0&\cdots&0&0&0&0\end{matrix}\right],$$
	with $m+2(m-1)=3(m-1)+1$ independent variables, whose display order is along rows, from left to right, so that, in particular, $\{x_{m-1,1}, x_{m,2}, x_{m,1}\}$ are last.
	Let $\xx:=\xx_{m,m-2,m-2}$ denote this ordered set of variables.
	We let $R=k[\xx]$ as before denote the corresponding ground polynomial ring.
	
	Note that this matrix has a built-in recurrence facet, namely, for every $2\leq t\leq m-1$,  the $t\times t$ submatrix from the top right corner of $\mathfrak{G}_{m,m-2,m-2}$ is $\mathfrak{G}_{t,t-2,t-2}.$ 
	Thus, setting
 $f_m=f_{m,m-2,m-2}:=\det \mathfrak{G}_{m,m-2,m-2}$, 
 Laplace expansion along the last row yields
\begin{equation}\label{fexpansion}
	f_{m}=(-1)^{m+1}x_{m,1}f_{m-1}+x_{m-1,1}x_{m,2}f_{m-2},
\end{equation}
with the understanding that $f_{m-1}=f_{m-1,m-3,m-3}$ and $f_{m-2}=f_{m-2,m-4,m-4}$.

Let $V(f_{m})^*$ denote the dual variety over ${\rm Proj}(R)$ in its natural embedding, as in Remark~\ref{definition_dual_variety}.

\begin{Proposition}\label{dim_dual_equal_case}
With the above notation, one has $\dim V(f_m)^*=2m-2$.
\end{Proposition}
\begin{proof}
Let $H(f_m)$
denote the Hessian matrix of $f_m\in  R=k[\xx]$.
	
By \cite[Theorem 2.7]{DetBook2024}, one is to show that the rank of $H(f_m)$ module $f_m$ is $ 2m.$ 
 Since $\mathfrak{G}_{m,m-2,m-2}$ is a coordinate sparse section, by  \cite[Proposition 2.14]{DetBook2024}, it suffices to show that $\rank H(f_m)\,({\rm mod}f_m)\geq 2m,$ i.e., to find a $2m\times 2m$ submatrix $A$ of  $H(f_m)$ such that $f_m$ does not divide  $\det A$ as elements of $R$.
 
We induct on $m.$ For $m=2$,  $\mathfrak{G}_{m,m-2,m-2}$ is the $2\times 2$ generic matrix, in which  case the result is trivial.  

Suppose that $m\geq 3$. Let $D=\xx\setminus \{x_{m,1}\}$. 
Drawing upon \eqref{fexpansion} one sees that the $3(m-1)\times 3(m-1)$ submatrix $\left(\frac{\partial^2f_m}{\partial x_{i,j}\partial x_{i',j'}}\right)_{x_{i,j},x_{i',j'}\in D}$ of the Hessian matrix $H(f_m)$ has the following format 
\begin{equation}\label{format_sub}
	\left[
	\begin{array}{cc|ccccc}
		&x_{m,1}H(f_{m-1})+x_{m,2}x_{m-1,1}H (f_{m-2})&x_{m-1,1}\, \Theta (f_{m-2})^t&x_{m,2}\, \Theta (f_{m-2})^t\\ [4pt]
		\hline\\ [-8pt]
		&x_{m-1,1}\,\Theta (f_{m-2})&0&f_{m-2}\\
		&x_{m,2}\, \Theta (f_{m-2}) &f_{m-2}&0\\
	\end{array}
	\right].
\end{equation}
where  $\Theta (f_{m-2})$ stands for the Jacobian matrix of $f_{m-2}$ with respect to the variables of the set $D\setminus\{x_{m-1,1}, x_{m,2}\}.$

Now by the inductive step, there is a  $2(m-1)\times 2(m-1)$ submatrix $A_{m-1}$ of $H(f_{m-1}))$ such that $f_{m-1}$ does not divide $\det A_{m-1}.$ 
Then, from \eqref{format_sub}, there is a suitable $2(m-1)\times 2(m-1)$ submatrix $B_{m-1}$ of $H(f_{m-2})$, and a suitable $1\times 2(m-1)$  submatrix $C_{m-1}$ of  $\Theta (f_{m-2})$, such that the following is a $2m\times 2m$ submatrix of $H(f_m)$
$$A=\left[
\begin{array}{cc|ccccc}
	&&&x_{m-1,1}C_{m-1}^t&\\
	&x_{m,1}A_{m-1}+x_{m,2}x_{m-1,1}B_{m-1}&\\
	&&&x_{m,2}C_{m-1}^t&\\
	\hline
	&x_{m-1,1}C_{m-1}&0&&f_{m-2}\\
	&x_{m,2}C_{m-1}&f_{m-2}&&0\\
\end{array}
\right].
$$
The proof will be over pending the following

\smallskip

{\sc Claim.} $f_m$ does not divide $\det A.$

Indeed, supposing otherwise, let
$$\det A =hf_m=h((-1)^{m+1}x_{m,1} f_{m-1}+x_{m-1,1}x_{m,2}f_{m-2}),$$
for certain $h\in R.$
But, since 
$$A\equiv\left[
\begin{array}{cc|ccccc}
	&&&\boldsymbol0&\\
	&x_{m,1}A_{m-1}&\\
	&&&\boldsymbol0&\\
	\hline
	&\boldsymbol0&0&&f_{m-2}\\
	&\boldsymbol0&f_{m-2}&&0\\
\end{array}
\right]\mod (x_{m-1,1}, x_{m,2})$$ 
it follows that 
$$(-1)^{m+1}x_{m,1}hf_{m-1}\equiv -x_{m,1}^{2(m-1)}f_{m-2}^2\det A_{m-1}\mod (x_{m-1,1}, x_{m,2}).$$
Since neither side of this  modular equality involves either $x_{m-1,1}$ or $ x_{m,2}$, it becomes an equality over $R$.
But, then it is clearly an absurd because $f_{m-1}$ is an irreducible polynomial that does not divide any factor of the product $x_{m,1}^{2(m-1)}f_{m-2}^2\det A_{m-1}.$ 
\end{proof}

The above result can be extended to a suitable deformation of $\mathfrak{G}_{m,m-2,m-2}$ in the following sense.
Namely, let $L=(\ell_{i,j})$ denote a {\em coordinate} sparse section  of the $m\times m$ generic matrix $\XX$ specializing coordinate wise to $\mathfrak{G}_{m,m-2,m-2}.$

Let $\tilde{\XX}\subset \XX$ denote the subset of nonzero entries of $L$, and let as above $\xx$ denote the set of nonzero entries of  $\mathfrak{G}_{m,m-2,m-2}$. 
Consider the ideal $\fn:=(x_{i,j}\,|\,x_{i,j}\in \tilde{\XX}\setminus \xx)\subset k[\tilde{\XX}]$.

\begin{Proposition}
	With the above notation, if $L \, ({\rm mod}\, \fn) = \mathfrak{G}_{m,m-2,m-2}$ then
	$f:=\det L$ is a nonzero irreducible polynomial and $\dim V(f)^*=2m-2$.
\end{Proposition}
\begin{proof}  
	By hypothesis, $f_m:=f_{m,m-2,m-2}$ is the image of $f$ by the natural surjection $k[\tilde{\XX}]\surjects k[\tilde{\XX}]/\fn.$ Thus, since $f_m$ is irreducible, so is $f$.
	
		For the dual variety, as above, by \cite{Segre_Hessiano} it is equivalent to show that $\rank H(f)\,({\rm mod} f)=2m$.

Write $f=f_m+T$.
where $T\subset \fn.$
Thus,
$$ \left(\frac{\partial^2 f}{\partial x_{i,j}\partial x_{i',j'}}\right)_{x_{i,j},x_{i',j'}\in\xx}\equiv H(f_m) \mod \fn.$$
Therefore, to conclude that $\rank H(f) \,({\rm mod\,}f)\geq 2m$ it is enough to show that the rank of the modular Hessian matrix $H(f_m) \,({\rm mod\,}f_m)$ is at least $2m$.
But, this follows from Proposition~\ref{dim_dual_equal_case} and \cite{Segre_Hessiano}.
\end{proof}

 \section{Hilbert--Burch like matrices}\label{HB}
 
 In this section we consider an $n\times (n-1)$  matrix $\phi$ with homogeneous entries in a standard graded polynomial ring $R$ over an infinite field.
 Upon the usual condition that the ideal $I_{n-1}(\phi)$ have height $2$ (maximal possible), we are dealing with the syzygy matrix of  a perfect ideal of height $2$.
 We recall a few features of these ideals.

 \subsection{$3\times 2$ matrices having homogeneous  columns}

 Binary $3\times 2$ matrices have been the subject of many authors (see \cite{BuJou}, \cite{Cox07}, \cite{CHW}, \cite{Syl1}, to mention a few).
 We henceforth take up the ternary case.
 Thus,  assume that  $R=k[x,y,z]$, endowed with the standard grading in which $\deg x=\deg y=\deg z=1$.
 The case where the entries of the matrix are linear forms has been considered, even for $m\times (m-1)$ matrices, with arbitrary $m$ (see, e.g., \cite{SymbPowBir2014}, \cite{Lan}, \cite{linpres2018}).
 For entries of higher degrees, one or another among the standing hypotheses of \cite[Theorem 1.2]{MorUl1996} are not available.
 Thus, believing that most is known in this low dimension, may be delusional. 
 Even in this narrow environment, one topic that seems to require further non-trivial work is the structure of the associated algebras, such as the symmetric and Rees algebras, as well as the special fiber (fiber cone)

 To move on, let $\phi=(a_{i,j})_{1\leq i\leq 3\atop {1\leq j\leq 2}}$ denote a $3\times 2$ matrix over $R$, where $1\leq \deg a_{i,1}:=d_1\leq d_2:=\deg a_{i,2}$, for all $i$.
 Thus, the ideal $I:=I_2(\phi)\subset R$ of $2$-minors is a homogeneous ideal equigenerated in degree $d:=d_1+d_2$.
 We assume in addition that this ideal has codimension at least two (hence, exactly two). Therefore, $I$ is a perfect ideal, i.e., $R/I$ is Cohen--Macaulay.
 
Since obviously $I\subset I_1(\varphi)$, then $I_1(\varphi)$ has codimension $\geq 2$. 
If  $I_1(\varphi)$ has codimension $3$ then $I$ is an ideal of linear type (\cite[Lemma 3.1]{NejadSimis2011}) -- and actually must have a syzygy whose coordinates form a regular sequence (\cite[Theorem 2.1]{Toh2013}). Thus, the structure of the ring $R/I$ and its main associated algebras is well-known (the symmetric algebra of $I$ is a complete intersection).

Therefore,  assume henceforth that $I_1(\varphi)$ has codimension two. In particular, every minimal prime of $R/I_1(\varphi)$ is a minimal prime of $R/I$ but typically not the other way around.

In addition, set $\fm=(x,y,z)$ and $\xx=\{x,y,z\}$.
Mapping a polynomial ring $S=R[T_1, T_2, T_3]$ onto the Rees algebra $R[It]\subset R[t]$ of $I$, let $\mathcal{I}\subset S$ denote the corresponding kernel -- often called the homogeneous defining ideal of  $R[It]$.
Notable is the fact that, since $I$ is equigenerated, then $R[It]$ admits a natural bigrading induced by the standard bigrading of the polynomial ring $S=k[x,y,z,T_1,T_2,T_3]$.
Finally, let $C(\mathcal{I})$ stand for the homogeneous ideal of $R$ generated by the $\xx$-coefficients of every element of $\mathcal{I}$.

One is interested in the following questions:
 \begin{Questions}\label{q2}
	\rm
	
	(a) When is $C(\mathcal{I})=\fm$?
	
	(b) When is $C(\mathcal{I})$ an $\fm$--primary ideal?
	
	(c) When is the analytic spread $\ell(I)$ of $I$ maximum ($=3$)?
	
	(d) When is  the rational map
	${\mathbb P}^2 \dasharrow {\mathbb P}^2$ defined by generators of $I$ 
	birational?
\end{Questions}

Trivially, affirmative (a) $\Rightarrow$ affirmative (b). Also, since $I$ is equigenerated, then $\ell(I)$ is the Krull  dimension of the $k$-algebra $k[I]\simeq k[It]$, hence $\ell(I)\leq 2$ implies a nonzero strict polynomial relation, i.e., $C(\mathcal{I})=(1)$.
Therefore, affirmative (b) $\Rightarrow$ affirmative  (c). Moreover, clearly affirmative (d) $\Rightarrow$ affirmative (c).
Finally, affirmative (a) is a good deal in the direction of proving that $R[It]$ is regular locally in
codimension one. Recall that the latter is equivalent to birationality when $\dim R/I=0$, but certainly not in the present context with $\dim R/I=1$.

We now focus on the degrees $d_1\leq d_2$.

\smallskip

$\bullet$ {\large ${d_1=1}$}

 Thus,  the entries of the first column of $\phi$ are linear forms.
 By our standing assumption on the codimension of  $I_1(\varphi)$, the three forms are $k$-linearly dependent.
 By an elementary row operation (over $k$) -- which implies a $k$-linear change of the generators of $I$ -- we may assume that the first column has the shape $(\ell_1,\ell_2, 0)^t$, where $\ell_1,\ell_2$ are $k$-independent linear forms.
 
 Since $(\ell_1,\ell_2)$ is a prime ideal, we must have  $I_1(\varphi)=(\ell_1,\ell_2)$.
 In particular $I_1(\varphi)$ is an unmixed ideal. Further, by a change of variables, we may as well assume that $\ell_1=x, \ell_2=y$.

 Thus, we may assume that the matrix is of the form
 \begin{equation}\label{linear_syzyzy}
 	\left[\begin{matrix} x & \gamma_1 \\
 		y & \gamma_2 \\
 		0 & \gamma_3
 	\end{matrix}\right] ,
 \end{equation}
 where $\gamma_1, \gamma_2,\gamma_3\in (x,y)$  since $I_1(\phi)=(x,y)$.
Thus, $I=(x\gamma_3,y\gamma_3, x\gamma_2-y\gamma_1).$
Since we are assuming that $I$ has codimension two, then $\gcd (\gamma_3, x\gamma_2-y\gamma_1)=1$.
Moreover, we will assume that one at least among $\gamma_1, \gamma_2, \gamma_3$ involves effectively $z$.

Note that $I$ is generically a complete intersection except at $(x,y)$ and $I:(x,y)^{\rm sat}=J:(x,y)^{\rm sat}$, where $J:=(\gamma_3, x\gamma_2-y\gamma_1)\supset I$ admits same associated primes as $I$.

By \cite[Lemma 1.2 and Proposition 1.6]{De_Jonq} and the terminology there, one has:
 \begin{Proposition}
 	With the above notation and assumptions$:$
 	\begin{enumerate}
 	\item[{\rm (i)}]	 The rational map $\mathfrak{R}: \pp^2\dasharrow \pp^2$ defined by the generators of $I$ is confluent with the identity map of $\pp^1$.
 	\item [{\rm (ii)}]	 
 	$\mathfrak{R}$ is birational if and only if $\gamma_1, \gamma_2, \gamma_3$ are $z$-monoids, in which case the map is a de Jonqui\`eres map.
 \end{enumerate}
 \end{Proposition}

\begin{Remark}\label{Rees_algebra_downgraded}\rm
In addition to the above proposition, one has a full description of the presentation ideal of the Rees algebra of the ideal $I$ involving a so-called downgraded sequence of biforms (\cite[Theorem 2.6]{De_Jonq}).
In the case of cubics -- i.e., when $\deg \gamma_1= \deg \gamma_2=\deg \gamma_3=2$ -- the defining ideal of the Rees algebra has been previously noted in \cite{Vasc-Hong}.
\end{Remark}

$\bullet$ {\large ${d_1\geq 2}$} (hence, $d_2\geq 2$)

 Since the rational map defined by the $2$-minors of $\phi$ is  no longer automatically confluent with the identity map of $\pp^1$, in particular there is no obvious analogue of a  downgraded sequence as mentioned in Remark~\ref{Rees_algebra_downgraded} and, consequently, no predictability for the minimal number of generators of the homogeneous defining  ideal of the Rees algebra of the ideal $I$.

{\sc Sylvester forms.}
We resort to the more encompassing notion of a Sylvester form. 

For the basic facts on this notion we refer to either \cite[Section 4.1]{Syl2} or \cite[Section 2.1]{SimToh2015}. 
In particular, the references \cite{Syl1}, \cite{HasSim2012}, \cite{SimToh2015}, \cite{BuSiTo2016}, \cite{LinShen} contain an enormous amount of results as  to whether (or when), given a reasonably structured ideal, the homogenous defining ideal of its Rees algebra is generated via iterated Sylvester forms.

We emphasize that, though parts of the above references deal quite a bit with the case where $R/I$ is Artinian (i.e., $I$ is $\fm$-primary), here we stick to the case where $\dim R/I=1$, which seems to be a significant difference.

The following examples illustrate this line of approach.

\begin{Example}\label{deg4}\rm
Let
\[ \varphi = \left[ \begin{matrix}
	x^2 & y^2 \\
	y^2 & xy \\
	xz & x^2
\end{matrix}\right].
\]
\end{Example}
Clearly, $I_1(\varphi)=(x^2,xy,xz,y^2)$ has an embedded component, while its codimension two component is $(x,y^2)$. 
Let $t,u,v$ denote presentation variables over $R$ for the Rees algebra of $I=I_2(\phi)$.
Let $\mathcal{I}\subset R[t,u,v]$ as before denote the corresponding presentation ideal.
Since the forms $f:=x^2t+y^2u+xzv, g:=y^2t+xyu+x^2v$ are induced by syzygies they belong to $\mathcal{I}$, their Sylvester form with respect to $(x,y^2)$ is $\det \Psi$, where $\Psi$ is the $2\times 2$ matrix below:
$$\left[f\; g\right]=\left[\begin{matrix}
	xt+zv & u\\yu+xv & t
\end{matrix}
\right] \cdot \left[\begin{matrix}
	x\\
	y^2
\end{matrix}
\right].$$
As one verifies, the ideal $(f,g, \det\Psi)$ is the ideal of $2$-minors of the matrix
\[ \left[ \begin{matrix}
-t & yu+xv\\
	u  &-xt-zv \\
	x & y^2
\end{matrix} \right].
\]
Thus,  this ideal is Cohen--Macaulay.
A run with \cite{M2} shows that the ideal $(f,g, \det\Psi)$ is prime (though $R[t,u,v]/(f,g, \det\Psi)$ is not normal) hence must be the presentation ideal $\mathcal{I}$ of the Rees algebra.

\begin{Example}\label{deg4_bis}\rm
Let
\[ \varphi = \left[ \begin{matrix}
x^2 & yz \\
xy & y^2 \\
y^2 & xz
\end{matrix} \right].
\]
\end{Example}
In this case, the codimension two component of $I_1(\phi)$ is $(x,y)$.
The Sylvester form of $f,g$ with respect to $(x,y)$ is 
$$\det\Psi=
 \left[\begin{matrix} zt+yu & zv \\
	xu+yv & xt \end{matrix}\right]
=xzt^2+xytu-xzuv-yzv^2.
$$
Clearly, $\mathcal{J}:=(f,g,xzt^2+xytu-xzuv-yzv^2)\subset \mathcal{I}$ is not a prime ideal since it is contained in $(x,y)R[t,u,v]$.
A run of \cite{M2} reveals that $\mathcal{J}$ is nevertheless radical and Cohen--Macaulay, $\mathcal{J}:(x,y)$ contains the $(1,3)$-form $h:=zt^3+yt^2u-(x+z)tuv-yu^2v+zv^3$ and, in fact, $\mathcal{I}=(\mathcal{J},h )$.
As it comes out, $\mathcal{I}$ is not Cohen--Macaulay since the syzygy matrix of $\mathcal{J}$ extends by zeros to minimal syzygies of the former.

\begin{Remark}\rm
(i) Both Example~\ref{deg4} and Example~\ref{deg4_bis} have the same Hilbert function.

(ii) In  Example~\ref{deg4_bis}, $h$ is also a Sylvester form, this time around of $\{f, \det\Psi\}$ with respect to the forms $\{xz, y\}$. Note that $(y,z)$ is not a minimal prime of $I_1(\phi)$ but is one of  $I$.
\end{Remark}

Next is a list of questions, apparently open  even in the case where the entries of  $\phi$ are monomials and have standard `neighboring' degrees  $\left \lfloor{d/2}\right \rfloor$ and $\left \lceil{d/2}\right \rceil$, respectively, where $d$ is the common degree of the three minors.
Set $I=I_2(\phi)$ and assume as above that $I_1(\phi)$ has codimension two and, in addition, its codimension two primary component is a complete intersection (of two forms).
Finally, we keep the assumption that $d_2\geq d_1\geq 2$.

\begin{Question}\rm
When is the Rees algebra of $I$  Cohen--Macaulay? (It is always almost Cohen--Macaulay)
\end{Question}

\begin{Question}\rm
Is $\mathcal{I}$ minimally generated by the two syzygetic relations $\{f,g\}$ of $I$ and by iterated Sylvester forms starting with $\{f,g\}$?
 Moreover, do the minimal generating Sylvester forms have different bidegrees, that is, for each occurring bidegree there is only one minimal generator of this bidegree?
\end{Question}

\begin{Question}\rm
Is the rational map defined by the standard generators of $I_2(\phi)$  ever birational? (Possibly, almost never).
\end{Question}
We note that, if the last assertion of the previous question holds true, then the last question has a negative answer (see \cite[Theorem 3.2.22]{GradedBook})

\begin{Remark}
	\rm There is a subtle balance involving the various data that have impact on the above questions. Besides the standing assumption that  the codimension two primary component of $I_1(\phi)$ is a complete intersection, one has to look at the minimal primes of $I$ and the minimal primes of the presentation ideal of the symmetric algebra -- these are the ones that may be used to get Sylvester forms. There are easy examples in degrees $d_1=2,d_2=3$ in which $\mathcal{I}$ has a unique Sylvester form and, moreover, it has bidegree $(2,2)$, while it is Cohen--Macaulay.
\end{Remark}
 
 \subsection{Generating in two degrees}

In this part we take over the ``next'' case, namely, assume that the homogeneous height two perfect ideal is generated in two different degrees.
We state the preliminaries in terms of the associated Hilbert--Burch like matrix.

 \subsubsection{Setup}\label{Basic}

	Let $R$ be a standard graded polynomial ring $R=k[x_1,\ldots,x_d]$ over a field and let $n\geq 2$ be an integer with a partition $n=a+(n-a)$, where $1\leq a\leq n-1$. Given integers $\epsilon_1, \epsilon_2\geq 1$, introduce the $n\times (n-1)$ block matrix
	\begin{equation}\label{basic_matrix}
		\phi=\left[\begin{matrix}
			\,	\Phi_1\,\\
			\hline \,
			\,	\Phi_2 \,
		\end{matrix}\right],
	\end{equation}
	where $\Phi_1$ is an $a\times (n-1)$ submatrix  whose row entries have degree $\epsilon_1$ and $\Phi_2$ is an $(n-a)\times (n-1)$  matrix whose row entries  have degree $\epsilon_2$.
	We assume that the ideal $I_{n-1}(\phi)\subset R$ of maximal minors has codimension $\geq 2$ (hence, $2$).
	
	Set $I:=I_{n-1}(\phi)$. Since $n-a\leq n-1$,  the minimal graded free resolution of $I$ has the form 
	\begin{equation}\label{two-degree-resolution}
		0\to R(-D)^{n-1}\stackrel{\phi}\lar  R(-(D-\epsilon_2))^{n-a}\oplus R(-(D-\epsilon_1))^a\to I\to 0,
	\end{equation}
	where $D: =a\epsilon_1 + (n-a)\epsilon_2$.

To fix ideas, one might assume that $\epsilon_1\geq \epsilon_2$, so that the initial degree of $I$ turns out to be $D-\epsilon_1$. Howver, this is pretty much irrelevant in the present discussion.

\subsubsection{Minors fixing a submatrix}

Let $J\subset I$  stand for the subideal generated in  degree $D-\epsilon_1$. In other words, $J$ is the ideal generated by the maximal minors $f_1,\ldots,f_a$ of $\phi$ that fix the submatrix $\Phi_2$. 
A free resolution of $R/J$ was obtained formerly in \cite[Theorem D]{AnSi1986}, which holds regardless of grading considerations  (similar result for arbitrary matrices and sizes remains apparently open -- see \cite[Conjecture 6.4.15]{SimisBook}).

Here we make a digression, by emphasizing the role of the $R$-module $I/J$.
For this, assume that $J$ has codimension two as well, which is the typical assumption of  \cite[Theorem D]{AnSi1986}.
One can see in this case that $I/J$ is isomorphic to $\coker \mathcal{L}$ (more exactly, $I/J \simeq \coker \mathcal{L}(-(D-\epsilon_2))$ is a homogeneous isomorphism), hence its free $R$-resolution relates to the one of $R/J$, as in \cite[Thorem D]{AnSi1986}. For the reader's convenience we reformulate the argument of  \cite[Theorem D, (ii) $\Rightarrow$ (i)]{AnSi1986} in the present environment.

\begin{Proposition}\label{resij_implies_resrj} Let 
	$$0\to \bigoplus_{j} R(-j)^{\beta_{r,j}}\stackrel{\psi_{r-1}}\lar\cdots\stackrel{\psi_2}\lar\bigoplus_j R(-j)^{\beta_{2,j}}\stackrel{\psi_1}\lar R(-D)^{n-1}\stackrel{\Phi_2}\lar R^{n-a}(-(D-\epsilon_2))\to I/J\to 0$$
	be a minimal graded free resolution of the $R$-module $I/J.$ Then, the minimal graded free resolution of $R/J$ is
	{\small\begin{equation}\nonumber
			0\to \bigoplus_{j} R(-j)^{\beta_{r,j}}\stackrel{\psi_{r-1}}\lar\cdots\stackrel{\psi_2}\lar\bigoplus_j R(-j)^{\beta_{2,j}}\stackrel{\Phi_1\psi_1}\lar R(-(D-\epsilon_1))^{a}\stackrel{\mathfrak{f}}\to R\to  R/J\to 0,
	\end{equation}}
	where $\mathfrak{f}=\left[\begin{matrix}
		f_1& \cdots&f_a\end{matrix}\right].$
	In particular, $\depth R/J=\depth I/J$.
\end{Proposition}
\begin{proof} To prove this proposition it is sufficient to show that $$\ker(\mathfrak{f})=\Image(\Phi_1\,\psi_1)\quad \mbox{and}\quad\ker (\Phi_1\,\psi_1)=\Image\psi_2.$$ 
	
	The first follows from the equivalences
	\begin{eqnarray*}
		\eta\in\ker(\mathfrak{f})&\Leftrightarrow&
		\left[\begin{matrix}
			\eta\\\boldsymbol0_{(n-a)\times 1}\end{matrix}\right]\in\Image \phi\\
		&\Leftrightarrow& \phi\,\uu=\left[\begin{matrix}
			\eta\\\boldsymbol0_{(n-a)\times 1}\end{matrix}\right], \;\mbox{for some } \uu\in R^{n-1}\\
		&\Leftrightarrow& \left[\begin{matrix}
			\Phi_1\,\uu\\ \Phi_2\,\uu\end{matrix}\right]=\left[\begin{matrix}
			\eta\\\boldsymbol0_{(n-a)\times 1}\end{matrix}\right],\;\mbox{for some } \uu\in R^{n-1}\\
		&\Leftrightarrow&\Phi_1\,\uu=\eta\,\,\mbox{and}\,\, \Phi_2\,\uu=0,\;\;\mbox{for some } \uu\in R^{n-1}\\\
		&\Leftrightarrow&\Phi_1\,\uu=\eta\,\,\mbox{and}\,\,\uu\in\Image{\psi_1}, \;\mbox{for some } \uu\in R^{n-1}\\
		&\Leftrightarrow& \eta\in\Image(\Phi_1\,\psi_1)
	\end{eqnarray*}
	For the second, we have $\Phi_1\,\psi_1\,\psi_2=\boldsymbol0.$
	Hence, $\Image(\psi_2)\subset\ker(\Phi_1\,\psi_1).$ Now, consider $\eta\in\ker(\Phi_1\,\psi_1).$ Then, $\Phi_1\,\psi_1\,\eta=0.$ Since $\ker\Phi_2=\Image (\psi_1),$ we have also $\Phi_2\,\psi_1=0.$ Thus, $$\phi\,\psi_1\eta=\left[\begin{matrix}
		\Phi_1 \\\Phi_2\end{matrix}\right]\psi_1\eta=0.$$
		 Since $\ker\phi=\boldsymbol0$ then $\psi_1\eta=\boldsymbol0$, hence in particular, $\eta\in \Image(\psi_2).$ Thus, $\ker (\Phi_1\,\psi_1)=\Image(\psi_2)$ as desired.
\end{proof}

\subsubsection{Revisiting the Buchsbaum--Rim complex}

We recapitulate in the present environment how the  minimal graded free resolution of the $R$-module $I/J$ is  determined.

Quite generally, let $R$ be a Noetherian ring and $\psi: F\to G$ be a map of free modules over a ring $R.$  Write $s:=\rank F$ and $r:=\rank G$ and suppose that $s\geq r.$ The {\it Buchsbaum-Rim complex} of $\psi$ has the following well-known shape
	\begin{equation}\label{BuRim}
		0\to ({\rm Sym}_{s-r-1} (G))^{\ast}\otimes \wedge^{s}F\stackrel{\partial}\lar ({\rm Sym}_{s-r-2}(G))^{\ast}\otimes \wedge^{s-1}F \stackrel{\partial} \to\cdots\stackrel{\partial}\lar \wedge^{r+1} F\stackrel{\eta}\lar F\stackrel{\Psi}\lar G.
	\end{equation}
 By choosing bases of the  free modules, we can partly make explicit the nature of its differentials (see \cite{Eis}[Appendix 2] for further details):

\begin{Remark}\label{diferentials}\rm
	Let $\{e_1,\ldots,e_s\}$ denote a free basis of $F.$ Given indexes $1\leq i_1<\ldots<i_{r+1}\leq s$, with $e_{i_1}\wedge\cdots\wedge e_{i_{r+1}}\in \bigwedge^{r+1}F$, then the map $\eta$ is given by
		$$
		\eta(e_{i_1}\wedge \cdots\wedge e_{i_{r+1}})= \sum_{j=1}^{r+1}(-1)^{r+1-j}\det(\phi_j) e_{i_j},$$
		where $\phi_j$ is the $r\times r$ submatrix of $\psi$ with columns corresponding to the basis elements indexed by $i_1<\ldots< \hat{i_j} <\cdots <i_{r+1}$. 
		
Moreover, the  entries of  any differential $\partial$ are $\mathbb{Z}$-linear forms in the entries of $\psi.$
\end{Remark}

\begin{Proposition}\label{BR_main}{\rm (\cite[Section A2.6]{Eis})}
	The Buchsbaum--Rim complex of $\psi$ is a free resolution of $\coker\psi$ if and only if $\grade I_r(\psi)\geq s-r+1.$
\end{Proposition}
At our end, the  interest lies in the graded case of the map $\psi$. Precisely, we consider the following landscape, for whose details we claim no priority:
\begin{Proposition}\label{BRimparticularcase}
	Let $R=k[x_1,\ldots,x_d]$ be a standard graded polynomial ring over a field $k$ and let $\psi$  $R(-\boldsymbol{\delta})^{s}\stackrel{\psi} \to R^{r}$ denote a graded  map satisfying $\Ht I_{r}(\psi)=s-r+1$. Write $M:=\coker\psi$.
	Then$:$
	\begin{enumerate}
		\item[{\rm (i)}] The minimal graded free resolution of $M$ has the form
			\begin{eqnarray*}
				0\to R(-s\boldsymbol{\delta})^{\beta_{s-r-1}}\stackrel{\partial}\lar R(-(s-1)\boldsymbol{\delta})^{\beta_{s-r-2}} \stackrel{\partial}\lar\cdots \to R(-(r+2)\boldsymbol{\delta})^{\beta_1} \\
				\stackrel{\partial}\lar R(-(r+1)\boldsymbol{\delta})^{\beta_0}\stackrel{\eta}\lar R(-\boldsymbol{\delta})^{s}\stackrel{\psi}\lar R^{r} \to M\to 0.
			\end{eqnarray*}
		where $\beta_i={r-1+i\choose i}{s\choose i+r+1}$ for each $0\leq i\leq s-r-1.$
		\item[{\rm (ii)}] If, moreover, $r=s-2,$ then the minimal graded free resolution of $M$ has the form 
		\begin{equation*}
			0\to R(-4\boldsymbol{\delta})^{s-2} \stackrel{\psi^t}\lar R(-3\boldsymbol{\delta})^{s} \stackrel{\eta}\lar R(-\boldsymbol{\delta})^{s}\stackrel{\psi}\lar R^{s-2} \to M\to 0,
		\end{equation*}
		where $\eta$ is the skew-symmetric matrix
		\begin{equation*}
			\eta=\left[\begin{matrix}
				0&(-1)^{s+r-2}\delta_{1,2}&\ldots&(-1)^{r}\delta_{1,s}\\
				(-1)^{s+r-1}\delta_{1,2}&0&\ldots&(-1)^{r-1}\delta_{2,s}\\
				\vdots&\vdots&\ddots&\vdots\\
				(-1)^{s-1}\delta_{1,s}&(-1)^{s-2}\delta_{2,s}&\ldots&0
			\end{matrix}\right].
		\end{equation*}
Here the non-null entries of the  $j$th column of $\eta$ are the ordered signed $(s-2)$-minors of the submatrix of $\psi$  obtained by omitting its $j$th column. In particular, one has a matrix equality 
		\begin{equation*}
			\TT\eta=	\left[\begin{matrix}
				\Delta_1&\cdots&\Delta_s\end{matrix}\right]
		\end{equation*}
		where $\TT=\left[\begin{matrix}
			T_1&\cdots&T_{s}\end{matrix}\right]$ is a vector of indeterminates over $R$ and $\Delta_1,\ldots,\Delta_s$ are the ordered signed $(s-1)$-minors of the $s\times (s-1)$ matrix $\left[\begin{matrix}
			\TT^t&\psi^t\end{matrix}\right].$
	\end{enumerate}
\end{Proposition}

\begin{proof}
	(i) Referring to (\ref{BuRim}, one has) $({\rm Sym}_{i}(G))^{\ast}\otimes \wedge^{i+r+1}F\simeq R^{{r-1\choose i} }\otimes R^{{s\choose i+r+1}}\simeq R^{\beta_i}.$ 
	Then Proposition~\ref{BR_main}  and Remark~\ref{diferentials} give the stated  minimal graded free resolution of $M$.
	
	(ii) As a special case of (i), the minimal graded free resolution of $M$ has the following shape 
	
	\begin{equation*}
		0\to R(-4\boldsymbol{\delta})^{s-2} \stackrel{\psi_1}\lar R(-3\boldsymbol{\delta})^{s} \stackrel{\eta}\lar R(-\boldsymbol{\delta})^{s}\stackrel{\psi}\lar R^{s-2} \to M\to 0.
	\end{equation*}
In order to make $\psi_1$ explicit, consider  respective free bases  $$\{(-1)^{s-1}e_2\wedge\cdots\wedge e_s,\ldots, (-1)^{s-i}e_1\wedge\cdots\wedge\widehat{e}_i\wedge\cdots\wedge e_s,\ldots,(-1)^{(s-s)}e_1\wedge\cdots\wedge e_{s-1}\}$$
	and $\{e_1,\ldots,e_s\}$ of  $ R(-(r+1)\boldsymbol{\delta})^s\simeq \wedge^{s-1} F(-(r+1)\boldsymbol{\delta})$   and $F(-\boldsymbol{\delta})=R(-\boldsymbol{\delta})^{s}$.
	Then, Remark~\ref{diferentials} implies the claimed format of $\eta$. Since  $\psi \eta=\boldsymbol0$ and $\eta$ is skew-symmetric then $\eta\psi^t=\boldsymbol0.$ Thus, we have the complex
	\begin{equation}\label{resMr=s-2}
		0\to R(-4\boldsymbol{\delta})^{s-2} \stackrel{\psi^t}\lar R(-3\boldsymbol{\delta})^{s} \stackrel{\eta}\lar R(-\boldsymbol{\delta})^{s}\stackrel{\psi}\lar R^{s-2} \to M\to 0.
	\end{equation}
Since $\Ht I_{r}(\psi)=3$,  the Buchsbaum-Eisenbud acyclicity criterion implies that \ref{resMr=s-2} is acyclic.
\end{proof}

\subsubsection{Back to Setup~\ref{Basic}}

We now rephrase some of the previous free resolutions in the case where $\Ht I_{n-a}(\Phi_2)=a$.

\begin{Corollary}\label{ResofJ}
	If $\Ht I_{n-a}(\Phi_2)=a$ then the respective minimal graded free resolution of $I/J$ and $R/J$ have the following form
	{\small
		\begin{eqnarray}
			0\to R(-((n-2)\epsilon_2+D))^{\beta_{a-2}}\to R(-((n-3)\epsilon_2+D))^{\beta_{a-3}}\to\cdots\to R(-((n-a+1)\epsilon_2+D))^{\beta_1} \nonumber\\\to R(-((n-a)\epsilon_2+D))^{\beta_0}\stackrel{\eta}\lar R(-D)^{n-1}\stackrel{\Phi_2}\lar R^{n-a}(-(D-\epsilon_2))\to I/J\to 0\nonumber
	\end{eqnarray}}
	and
	{\small
	\begin{eqnarray}\nonumber
		0\to R(-((n-2)\epsilon_2+D))^{\beta_{a-2}}\to R(-((n-3)\epsilon_2+D))^{\beta_{a-3}}\to\cdots\to R(-((n-a+1)\epsilon_2+D))^{\beta_1}\\ \nonumber
		\to R(-((n-a)\epsilon_2+D))^{\beta_0} \stackrel{\Phi_1\eta}\lar R(-(D-\epsilon_1))^{a}\to R\to R/J\to 0\nonumber
\end{eqnarray}}
	where   $\beta_i={n-d-1+i\choose i}{n-1\choose i+n-a+1}$ for  $0\leq i\leq a-2$.
\end{Corollary}
\begin{proof} Since  $\Ht I_{n-a}(\Phi_2) =a=(n-1)-(n-a)+1$ then, by Proposition \ref{BR_main} the minimal graded free resolution of $\coker\Phi_2$ is
	
	{\small
		\begin{eqnarray}
			0\to R(-(n-1)\epsilon_2)^{\beta_{a-2}}\to R(-(n-2)\epsilon_2)^{\beta_{a-3}}\to\cdots\to R(-(n-a+2)\epsilon_2)^{\beta_1} \nonumber\\\to R(-(n-a+1)\epsilon_2)^{\beta_0}\to R(-\epsilon_2)^{n-1}\stackrel{\Phi_2}\lar R^{n-a}\to \coker \Phi_2\to 0\nonumber
	\end{eqnarray}}
			But, in this case, $I/J\simeq \coker\Phi_2(-(D-\epsilon_2)) .$ Hence, the minimal graded free resolution of $I/J$ it is as stated. The minimal graded free resolution of $R/J$ follows from the minimal graded free resolution of $I/J$ and Proposition~\ref{resij_implies_resrj}.
		\end{proof}
		
		
		The hypothesis $\Ht I_{n-a}(\Phi_2)=a$ in the above corollary imposes $1\leq a\leq d,$ where $d=\dim R$. Obviously, $J$ is a principal ideal if $a=1$ and a complete intersection if $a=2.$ Thus, the interest lies on the range $3\leq a\leq d.$ 
		In this range, if $\Ht I_{n-a}(\Phi_2)=a$ then the ideal $J$ is not perfect since $\dim R/J=d-2$, while $\depth R/J=d- {\rm hd}_R(R/J)\leq d-3$.
		
		We tacitly assume that $\Ht I_{n-a}(\Phi_2)=a$ for the rest of the section.
		Throughout will consider the extreme values $a=3$ and $a=d$ in some more detail.
		
		\smallskip

		{\Large $\bullet \;a=3$}
		
		\smallskip

		The free resolution of the ideal $J$ follows from Corollary~\ref{ResofJ}. We next lok at its symmetric algebra.
		Letting $S=R[T_1,T_2,T_3]=R[\TT]$ be a standard polynomial ring over  $R=k[x_1,\ldots,x_d]$,  we know that the symmetric algebra ${\rm Sym}_R(J)$ of $J$ over $R$ is presented over $S$ by the ideal  $I_1(\TT \Phi_1\,\eta)$ generated by the entries of the matrix product $[T_1\; T_2\: T_3]\, \Phi_1\,\eta$. 
		\begin{Proposition} Assume that $a=3.$ If $\Ht I_{n-3}(\Phi_2)=3$ then the symmetric algebra ${\rm Sym}_R(J)$ is a Cohen-Macaulay $S$-module with  minimal  graded free resolution 
			{\small
				\begin{eqnarray*}
					0 &\to & S(-((n-3)\epsilon_2+2\epsilon_1+2)) \oplus S(-((n-2)\epsilon_2+\epsilon_1+1)))^{n-3}\\ &\to& S(-((n-3)\epsilon_2+\epsilon_1+1)^{n-1} \to S\to {\rm Sym}_R(J)\to0.
			\end{eqnarray*}}
		\end{Proposition}
		\begin{proof} (i) By the Huneke-Rossi formula \cite[Theorem 2.6]{HR} we have $$\dim{\rm Sym}_R(J) = \sup_{P\in\spec{R}}\{\mu(J_P)+\dim R/P\}.$$
			If $\Ht P\geq 2$ then $$\mu(J_P)+\dim R/P\leq 3+(d-2)=d+1.$$
			If $\Ht P=1$ then
			$$\mu(J_P)+\dim R/P= 1+(d-1)=d.$$
			If $P=(0)$ then
			$$\mu(J_P)+\dim R/P=d+1.$$
			Thus, $\dim{\rm Sym}_R(J)=d+1.$

			It follows that $\Ht I_1(\TT \Phi_1\,\eta)=2.$ By Proposition~\ref{BRimparticularcase}(ii),  $I_1(\TT \Phi_1\,\eta)$ is the ideal generated by the $(n-2)$-minors of the $(n-1)\times(n-2)$ matrix $\left[\begin{array}{c|c}
				(\TT \Phi_1)^t&\Phi_2^{t}\end{array}\right].$ Hence, by the Hilbert-Burch theorem, $I_1(\TT \Phi_1\eta)$ is a perfect ideal of height $2$ with free resolution
			$$0\to S^{n-2}
			\stackrel{[(\TT \Phi_1)^t |  \Phi_2^{t}]}{\longrightarrow} S^{n-1}\to I_1(\TT \Phi_1\eta)\to 0.$$
			Inspection of the  degrees in the involved matrices gives the remaining degrees of  the minimal graded free resolution of ${\rm Sym}_R(J)$ as stated.
		\end{proof}

		{\Large $\bullet \;a=d$ and $\epsilon_2=1$}
		
		\smallskip
		
		We assume throughout that $\Phi_2\neq 0$, i.e., that $\phi$ has effectively nonzero linear entries.
		In other words, $a\leq n-1$, and hence in the present case, $n-d\geq              1$.
		
		We need the following lemma, which runs independently and is mostly well-known:
		
		\begin{Lemma}\label{BR_basic_equivalences}
			Let $R=k[x_1,\ldots,x_d]$ denote as previously a standard graded polynomial ring of a field $k$ and let  $R(-1)^{n-1} \stackrel{\psi}{\to} R^{n-d}$ be a graded linear map. The following assertions are equivalent:
			\begin{enumerate}
				\item[{\rm(i)}] $I_{n-d}(\psi)=(x_1,\ldots,x_d)^{n-d}.$
				\item[{\rm(ii)}]  $\Ht I_{n-d}(\psi)=d.$
				\item[{\rm(iii)}]  The Buchsbaum--Rim complex of $\psi$ is acyclic.
				\item[{\rm(iv)}] The Eagon--Northcott complex of $I_{n-d}(\psi)$ is acyclic.
			\end{enumerate}
		\end{Lemma}
		\begin{proof}
			(i) $\Rightarrow$ (ii) This is obvious.
			
			(ii) $\Leftrightarrow$ (iii) This is Theorem~\ref{BR_main}.
			
			(ii) $\Rightarrow$ (iv) This is well-known.
			
			(iv) $\Rightarrow$ (i) Since the Eagon--Northcott complex is a minimal free graded resolution of $I_{n-d}(\psi)$, the minimal number of generators of the latter is ${{n-1}\choose {n-d}}$, which is the minimal number of generators of $(x_1,\ldots,x_d)^{n-d}.$
		\end{proof}

		\smallskip
		
		Recall the notion of {\em chaos invariant} from \cite[Section 2]{linpres2018}.

		\begin{Proposition} With the notation and assumptions of {\rm Setup~\ref{Basic}}, assume that $a=d.$\label{saturation_equivs}
			The following are equivalent:
			\begin{enumerate}
				\item[{\rm(a)}] $J^{\rm sat}=I.$
				\item[{\rm(b)}] $\Ht I_{n-d}(\Phi_2)=d.$
				\item[{\rm(c)}] The Buchsbaum--Rim complex of $\Phi_2$ is a graded minimal free resolution of $I/J$. 
			\end{enumerate}
			Moreover, if $d=3$ and $\epsilon_1 =1$ then these conditions are equivalents to$:$
			\begin{enumerate}
				\item[{\rm(d)}] The chaos invariant  of $I$ is $\geq n-3.$
			\end{enumerate}
		\end{Proposition}
		\begin{proof} First, recall that the matrix $\Phi_2$ is the syzygy matrix of the $R$-module $I/J.$ Thus, since the annihilator $0:_RI/J$ and the zeroth Fitting ideal $I_{n-d}(\Phi_2)$ of $I/J$ have the same radical then $$\Ht 0:_RI/J=\Ht I_{n-d}(\Phi_2).$$
			With this equality and the fact that $J^{\rm sat}= I$ if and only if $I/J$ is $(x_1,\ldots,x_d)$-primary  as an $R$-module, the equivalence of (a) and (b) follows through.
			The equivalence of (b) and (c) follows from Lemma~\ref{BR_basic_equivalences}.
			
			The equivalence of (d) and (b) under the stated hypotheses follows immediately from the definitions.
		\end{proof}
		
		Given an integer $s\geq 2$, an ideal $I\subset R$ in a Noetherian ring satisfies the property $G_s,$} if, locally at at any prime ideal $P\supset I$ such that $\Ht P\leq s-1$, the minimal number of generators of $I$ is at most $\Ht P$. 
	Clearly, if $t\leq s$ then  $G_s \Rightarrow G_t$.
	If this property is satisfied for $s\geq \dim R+1$ (hence, for every $s\geq 1$) then the condition has been variously called  $G_{\infty}$ or $F_1$.
	The condition  was originally introduced in \cite{Art-Nag1972} and has since been explored by many authors.
	
	\begin{Corollary}
		Assume the equivalent conditions of {\rm Proposition~\ref{saturation_equivs}}. Then$:$
		\begin{enumerate}
			\item[{\rm (a)}] $J$ satisfies $G_{\infty}$ if and only if $I$ satisfies $G_d.$
			\item[{\rm(b)}] If $\epsilon_1=\epsilon_2=1$ then $J$ is a reduction of $I.$
			\item[{\rm(c)}] If $d=3$ then $J$ is of linear type if and only if $I$ satisfies $G_3.$
		\end{enumerate}
	\end{Corollary}
	\begin{proof} (a) By hypothesis we have $J^{\rm sat}=I.$ Thus, $J_P=I_P$ for every prime ideal $P\neq (x_1,\ldots,x_d).$ With this and the assumption that $\mu(J)=d$ the stated equivalence follows through.
		
		(b) Letting $f_1,\ldots,f_n$ be the ordered signed $(n-1)$-minors of $\phi,$ one has $J=(f_1,\ldots,f_d).$
		Since the rows of $\phi$ are syzygies of  $f_1,\ldots,f_n$, we can write
		\begin{equation*}
			\left[\begin{matrix}
				f_{d+1}&\cdots&f_n
			\end{matrix}\right]\Phi_2=-\left[\begin{matrix}f_1&\cdots& f_d\end{matrix}\right]
			\Phi_1.
		\end{equation*}
Thus, for any  $(n-d)\times(n-d)$ submatrix $B$ of $\Phi_2$, one has
$$I_1(\left[\begin{matrix}
			f_{d+1}&\cdots&f_n
		\end{matrix}\right]B)\subset ( x_1,\ldots, x_d) J.$$
Bringing in the adjugate matrix, one can write $$I_1(\left[\begin{matrix}
			f_{d+1}&\cdots&f_n
		\end{matrix}\right] B\,{\rm adj}(B))\subset ( x_1,\ldots, x_d)^{n-d}J,$$
and hence, $(\det B)(f_{d+1},\ldots,f_n)\subset ( x_1,\ldots, x_d)^{n-d}J.$ Since $B$ is arbitrary, it obtains 
		\begin{equation*}\label{I_{n-d}IsubsetI_{n-d}J}
			I_{n-d}(\Phi_2)( f_{d+1},\ldots,f_n)=( x_1,\ldots, x_d)^{n-d}( f_{d+1},\ldots,f_n)\subset ( x_1,\ldots, x_d)^{n-d} J,
		\end{equation*}
where the equality follows from Lemma~\ref{BR_basic_equivalences} since $\Phi_2$ is assumed to be linear.

This inclusion is an expression of integral dependence.
In fact, let $M_1,\ldots,M_N$ denote the set of monomials of $R$  of degree $n-d.$ Then
 $$M_jf_i=L_{j,1} M_1+\cdots +L_{j,N} M_N, \quad 1\leq j\leq N, \, d+1\leq i\leq n,$$
for certain $k$-linear forms $L_{j,l}\in k[f_1,\ldots,f_d] \subset R$ over $f_1,\ldots,f_d$.
These relations can be read as
 $$\left(f_i\mathbb{I}_N- (L_{j,l})_{1\leq j,l\leq N}\right)\left[\begin{matrix}M_1&\cdots&M_N\end{matrix}\right]^t=\boldsymbol0,$$
or yet, 
		$$\det\left(f_i\mathbb{I}_N-(L_{j,l})_{1\leq j,l\leq N}\right)( x_1,\ldots,x_d)^{n-d}=0.$$
		Hence, 
		$$\det\left(f_i\mathbb{I}_N-(L_{j,l})_{1\leq j,l\leq N}\right)=0$$
which is an equation of integral dependence of $f_i$ over $J=(f_1,\ldots,f_d)$.  In particular, $J$ is a reduction of $I.$

		(c) Because of item (a), it follows from \cite[Proposition 5.5 and Proposition 9.1]{Trento} that the property $G_3$ implies the linear type property. The converse is clear.
	\end{proof}

  \bibliographystyle{amsalpha}


\noindent {\bf Addresses:}

\smallskip

\noindent {\sc Zaqueu Ramos} \hspace{9.5cm} {\sc Aron Simis}\\
Departamento de Matem\'atica, CCET \hspace{2cm} Departamento de Matem\'atica, CCEN\\ 
Universidade Federal de Sergipe  \hspace{3cm} Universidade Federal de Pernambuco\\ 
49100-000 S\~ao Cristov\~ao, SE, Brazil \hspace{4cm} 50740-560 Recife, PE, Brazil\\
\indent zaqueu@mat.ufs.br \hspace{7cm} aron.simis@ufpe.br\\

\end{document}